\documentclass{mcom-l}

\usepackage{amsmath, amssymb, amsthm, amsfonts, mathrsfs}

\usepackage{float}
\usepackage{color}
\usepackage{bm, bbm}
\usepackage{microtype}
\usepackage{fancyhdr}
\usepackage{caption}
\usepackage{wrapfig}
\usepackage{cases}
\usepackage{setspace}
\usepackage{graphicx}
\usepackage{subfigure}

\usepackage{epstopdf}
\usepackage{lineno}
\modulolinenumbers[5]
\usepackage{booktabs}
\usepackage[colorlinks,linkcolor=blue,anchorcolor=blue,citecolor=blue]{hyperref}
\usepackage[american]{babel}
\usepackage{microtype}
\usepackage{ragged2e}

\usepackage{tabularx} 
\usepackage{makecell} 
\usepackage{array} 
\usepackage{multirow} 
\usepackage{lipsum} 

\newtheorem{theorem}{Theorem}[section]
\newtheorem{lemma}[theorem]{Lemma}

\theoremstyle{definition}

\theoremstyle{proposition}
\newtheorem{proposition}[theorem]{Proposition}

\theoremstyle{corollary}
\newtheorem{corollary}[theorem]{Corollary}

\theoremstyle{remark}
\newtheorem{remark}[theorem]{Remark}

\numberwithin{equation}{section}
\allowdisplaybreaks[4]

\newcommand{\rd}{\mathrm d}
\newcommand{\hP}{\mathbb P}
\newcommand{\hR}{\mathbb R}
\newcommand{\cA}{\mathcal A}
\newcommand{\cB}{\mathcal B}
\newcommand{\cC}{\mathcal C}
\newcommand{\cF}{\mathcal F}
\newcommand{\cI}{\mathcal I}
\newcommand{\cJ}{\mathcal J}
\newcommand{\cK}{\mathcal K}
\newcommand{\cM}{\mathcal M}
\newcommand{\cO}{\mathcal O}

\begin{document}

\title[Numerical methods of stochastic Volterra equations]
{Error analysis of numerical methods on graded meshes for stochastic Volterra equations}


\author[X.\ Dai]{Xinjie Dai}
\address{School of Mathematics and Statistics, Yunnan University, Kunming, 650504, Yunnan, China; LSEC, ICMSEC, Academy of Mathematics and Systems Science, Chinese Academy of Sciences, Beijing 100190, China}
\curraddr{}
\email{xinjie@smail.xtu.edu.cn}
\thanks{This work is supported by National key R\&D Program of China (No.\ 2020YFA0713701), National Natural Science Foundation of China (Nos.\ 12031020, 11971470, 11871068), and China Postdoctoral Science Foundation (No.\ 2022M713313).}

\author[J.\ Hong]{Jialin Hong}
\address{LSEC, ICMSEC, Academy of Mathematics and Systems Science, Chinese Academy of Sciences, Beijing 100190, China; School of Mathematical Sciences, University of Chinese Academy of Sciences, Beijing 100049, China}
\curraddr{}
\email{hjl@lsec.cc.ac.cn}
\thanks{}

\author[D.\ Sheng]{Derui Sheng}
\address{LSEC, ICMSEC, Academy of Mathematics and Systems Science, Chinese Academy of Sciences, Beijing 100190, China; Department of Applied Mathematics, The Hong Kong Polytechnic University, Hung Hom, Kowloon, Hong Kong}
\curraddr{}
\email{sdr@lsec.cc.ac.cn (Corresponding author)}
\thanks{}


\date{}

\dedicatory{}


\begin{abstract} 
This paper presents the error analysis of numerical methods on graded meshes for stochastic Volterra equations with weakly singular kernels. We first prove a novel regularity estimate for the exact solution via analyzing the associated convolution structure. This reveals that the exact solution exhibits an initial singularity in the sense that its H\"older continuous exponent on any neighborhood of $t=0$ is lower than that on every compact subset of $(0,T]$. Motivated by the initial singularity, we then construct the Euler--Maruyama method, fast Euler--Maruyama method, and Milstein method based on graded meshes. By establishing their pointwise-in-time error estimates, we give the grading exponents of meshes to attain the optimal uniform-in-time convergence orders, where the convergence orders improve those of the uniform mesh case. Numerical experiments are finally reported to confirm the sharpness of theoretical findings. 
\end{abstract}

\maketitle


\section{Introduction} 

Stochastic Volterra equations are first studied by Berger and Mizel in \cite{BergerMizel1980} and are often applied to model the rough volatility in mathematical finance and the anomalous diffusion in statistical physics; see e.g., \cite{ElEuch2019, LiLiu2017} and references therein. We consider the construction and analysis of numerical methods for the following stochastic Volterra equation with weakly singular kernels
\begin{align} \label{eq.model}
x(t) = x_0 + \int_0^t (t-s)^{-\alpha} f(x(s)) \rd s + \int_0^t (t-s)^{-\beta} g(x(s)) \rd W(s), ~~ t \in [0,T].
\end{align}
Here, $\alpha \in (0,1)$, $\beta \in (0,\frac{1}{2})$, and $W$ is an $m$-dimensional standard Brownian motion defined on some complete filtered probability space $(\Omega, \cF, \{\cF_t\}_{t\in[0,T]}, \hP)$. Assume that $x_0: \Omega \rightarrow \hR^d$, $f: \hR^d \rightarrow \hR^d$ and $g: \hR^d \rightarrow \hR^{d \times m}$ satisfy suitable conditions such that \eqref{eq.model} admits a unique strong solution $x \in \cC([0,T],L^p(\Omega,\hR^d))$ with $p \geq 2$ (cf.\ \cite[Theorem 3.3]{AbiJaberLarssonPulido2019} and \cite[Theorem 2.5]{LiHuangHu2022}). 

The numerical study of stochastic Volterra equations has received much attention; see e.g., \cite{DaiXiao2020, DoanHuongKloedenVu2020, FangLi2020, LiHuangHu2022, NualartSaikia2022, RichardTanYang2021, RichardTanYang2023, Zhang2008} for the analysis of the convergence, stability and error distribution of numerical methods. Two issues that cannot be ignored are the memory of kernels and the low regularity of solutions, as they result in expensive computational costs and poor approximation accuracy, respectively. On one hand, due to the memory of kernels, the implementation of a single sample of the Euler--Maruyama (EM) method requires a computational cost of $\cO(N^2)$ with $N$ the total number of time steps. By the sum-of-exponentials approximation, the fast EM method is developed in \cite{DaiXiao2020, FangLi2020}, which can be seen as a Markovization of non-Markovian processes and reduces the computational cost to $\cO\big( N (\log N)^2 \big)$. On the other hand, owing to the singularity of kernels and the non-smoothness of Brownian motion, the solution to \eqref{eq.model} is only $\min\{1-\alpha, \frac{1}{2}-\beta\}$-H\"older continuous on $[0,T]$ (see \cite[Theorem 2.6]{LiHuangHu2022}), which leads to a convergence order $\min\{1-\alpha,\frac{1}{2}-\beta\}$ of the EM method in $\cC([0,T],L^p(\Omega,\hR^d))$ in the literature. To improve the approximation accuracy, the Milstein method is proposed in \cite{LiHuangHu2022, RichardTanYang2021} by adding the first-order Taylor expansion term. However, the existing convergence order $\min\{1-\alpha,1-2\beta\}$ of the Milstein method is not necessarily higher than that of the EM method. This is essentially caused by the upper bound $1-\alpha$ of the above H\"older continuity exponent originated from the singular kernel $(t-s)^{-\alpha}$ in \eqref{eq.model}. In this paper, we are devoted to developing efficient numerical methods for \eqref{eq.model} by revealing the effect of the singular kernel on the regularity of the exact solution.

For the deterministic case of \eqref{eq.model}, due to the singularity of the kernel, the first-order derivative of its solution behaves like $t^{-\alpha}$, and thus the solution has an initial singularity in the sense that it exhibits a lower regularity near the initial time $t = 0$ than when $t$ is away from $0$ (see e.g., \cite{Brunner2004, Stynes2017}). A similar initial singularity can be expected for \eqref{eq.model} when the drift term governs its regularity. But we are not aware of any relevant results in the literature. In fact, an extension of the deterministic case to the stochastic case is not straightforward, since the solution to \eqref{eq.model} is nowhere differentiable with respect to time. To reveal the initial singularity in the stochastic case, we propose a novel regularity estimate for the solution to \eqref{eq.model} as follows
\begin{align} \label{eq:DXJ}
\big\| x(t) - x(s) \big\|_{L^p(\Omega,\hR^d)}
\leq C s^{\beta - \frac{1}{2}} (t-s)^{\frac{1}{2}-\beta} \quad \forall~0<s \leq t \leq T 
\end{align}
via analyzing the associated convolution structure. This validates that the solution to \eqref{eq.model} may exhibit a lower regularity near the initial time $t = 0$, since it is $(\frac{1}{2}-\beta)$-H\"older continuous on every compact subset of $(0,T]$, but is only $\min\{1-\alpha, \frac{1}{2}-\beta\}$-H\"older continuous on any neighborhood of $t=0$. 

\begin{figure}[b]
\centering
\!\!\subfigure[$r = 1$]{
\includegraphics[width=0.9\linewidth]{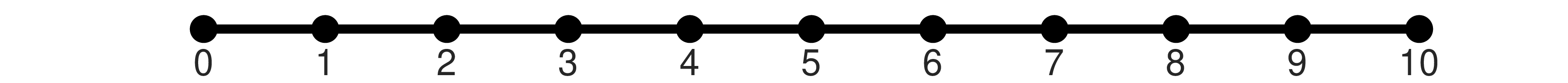}}
\\
\subfigure[$r = 2$]{
\includegraphics[width=0.9\linewidth]{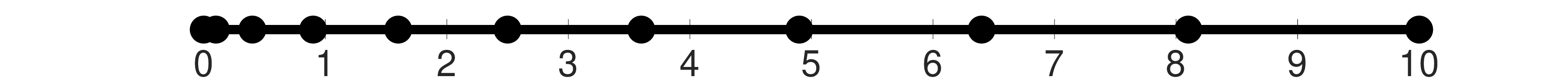}}
\caption{The meshes $\cM_1$ and $\cM_2$ with $T = 10$ and $N = 10$.}
\label{fig:sec1GM}
\end{figure}

Taking the above initial singularity into account, we study the EM method, fast EM method and Milstein method on graded meshes, hoping to attain higher uniform-in-time accuracy than using the uniform mesh in the literature. For integer $N \geq 1$, the graded mesh on $[0,T]$ is defined by 
\begin{align} \label{eq:defGM}
\cM_r := \big\{ t_n := T (n/N)^{r }, \, n= 0, 1, \ldots, N \big\}, 
\end{align} 
where $r \geq 1$ is called the grading exponent (cf. \cite{Brunner1985, Brunner2004}). Figure \ref{fig:sec1GM} plots the meshes $\cM_1$ and $\cM_2$ with $T = 10$ and $N = 10$. When $r = 1$, the mesh $\cM_1$ is uniform with stepsize $T / N$. When $r > 1$, the mesh $\cM_r$ is non-uniform, where the mesh points are densely clustered near the initial time $t = 0$. Making full use of \eqref{eq:DXJ}, we present the pointwise-in-time error analyses for the EM method, fast EM method and Milstein method on graded meshes in Theorems \ref{thm.EM}, \ref{thm.fastEM} and \ref{thm.Milstein}, respectively. These results clearly show the dependence of convergence orders of numerical methods on the singularity powers $\alpha,\,\beta$ and the grading exponent $r$, and indicate that the convergence order at the terminal time $T$ of the Milstein method is strictly higher than that of the EM method even for the uniform mesh case. Further, we give the optimal grading exponent of meshes in Corollary \ref{cor.EMres} (resp. Corollary \ref{cor.Milsteinres}) such that the convergence order of the EM method (resp.\ the Milstein method) can achieve $\frac{1}{2}-\beta$ (resp.\ $\min\{\frac{3}{2}-\alpha-\beta, 1-2\beta\}$) uniformly in time. To the best of our knowledge, this is the first result on applying graded meshes to numerically solve stochastic Volterra equations, the validity and superiority of which are supported by our theoretical results. In contrast to Milstein methods in the literature for stochastic Volterra equations, we mention that the one in the present work does not involve the first-order derivative of the drift coefficient $f$, and is thereby compatible with that of stochastic differential equations (cf.\ \cite{KloedenPlaten1992, MilsteinTretyakov2004}). 

The rest of the paper is organized as follows. The regularity theory of the solution to \eqref{eq.model} is presented in section \ref{sec.Regularity}. Then we study the EM method, fast EM method, and Milstein method on graded meshes in section \ref{sec.NumericalMethods}. Numerical tests are performed in section \ref{sec.NumerExper} to verify the sharpness of the theoretical results.

\textbf{Notations.} Use $C$ as a generic constant, which is always independent of $N$, but may be different when it appears in different places. Let $|\cdot|$ represent both the Euclidean norm on $\mathbb{R}^d$ and the Frobenius norm on $\mathbb{R}^{d \times m}$. The ceil function is denoted by $\lceil \cdot \rceil$. For integer $k \geq 1$, denote by $\cC_b^k(U, V)$ the space of not necessarily bounded functions from $U$ to $V$ that have continuous and bounded derivatives up to order $k$. Here, $U$ and $V$ are two Euclidean spaces (e.g., $\hR^d$ and $\hR^{d \times m}$).

\section{Regularity theory}
\label{sec.Regularity}

In this section, we are devoted to developing a novel regularity estimate (i.e., Theorem \ref{thm.newRegu}) to reveal the singularity of the solution to \eqref{eq.model} near the initial time $t = 0$. We begin with revisiting the existence, uniqueness, and H\"older continuity of the solution to \eqref{eq.model}. 

\begin{theorem} \textup{(\cite[Theorems 2.5 and 2.6]{LiHuangHu2022})}
\label{thm.wellpos}
Assume that $x_0 \in L^p(\Omega,\hR^d)$ with $p \geq 2$, and there exists $L > 0$ such that
\begin{align} \label{eq.assfLip}
|f(y) - f(z)| + |g(y) - g(z)| \leq L|y-z| \quad \forall\, y,z \in \hR^d.
\end{align}
Then \eqref{eq.model} admits a unique strong solution $x \in \cC([0,T], L^p(\Omega,\hR^d))$. Moreover,
\begin{align} \label{eq.xHolder}
\| x(t) - x(s) \|_{L^p(\Omega,\hR^d)} \leq C(t-s)^{\min\{1-\alpha,\, \frac{1}{2}-\beta\}} \quad \forall\, 0 \leq s < t \leq T,
\end{align}
where $C>0$ is independent of $t$ and $s$.
\end{theorem}

For measurable functions $F: [0,T]\times\Omega \rightarrow \hR^d$ and $G: [0,T]\times\Omega \rightarrow \hR^{d \times m}$, once we define the convolution integrals
\begin{align*} 
\Lambda_F(t) := \int_0^t (t-s)^{-\alpha} F(s) \rd s ~~ \mbox{and} ~~ \Upsilon_G(t) := \int_0^t (t-s)^{-\beta} G(s) \rd W(s) ~~ \forall\, t \in [0,T],
\end{align*}
the equation \eqref{eq.model} can be rewritten as
\begin{align} \label{eq.equivModel}
x(t) = x_0 + \Lambda_{f \circ x}(t) + \Upsilon_{g \circ x}(t).
\end{align} 
Here, by the proof of Theorem 2.6 of \cite{LiHuangHu2022}, for any $0 \leq s < t \leq T$, 
\begin{align}
\left\| \Lambda_{f \circ x}(t) - \Lambda_{f \circ x}(s) \right\|_{L^p(\Omega,\hR^d)}
&\leq C(t-s)^{1-\alpha}, \label{eq.drift} \\
\left\| \Upsilon_{g \circ x}(t) - \Upsilon_{g \circ x}(s) \right\|_{L^p(\Omega,\hR^d)}
&\leq C(t-s)^{\frac{1}{2}-\beta}. \label{eq.duff}
\end{align}

In general, the error analysis of numerical methods more or less involves the regularity of the underlying model. It is worth emphasizing that the H\"older regularity estimate \eqref{eq.xHolder} on $[0, T]$ is sharp for the exact solution (see \cite[Remark 2.7]{LiHuangHu2022}). In order to provide delicate error analyses for numerical methods of \eqref{eq.model}, we propose a novel regularity estimate for the exact solution, as stated in the following theorem.

\begin{theorem} \label{thm.newRegu}
Under the assumptions of Theorem \ref{thm.wellpos}, 
\begin{align*}
\| x(t) - x(s) \|_{L^p(\Omega,\hR^d)}
&\leq C s^{\beta-\frac{1}{2}} (t-s)^{\min\{\frac{3}{2}-\alpha-\beta,\, 1\}} + C(t-s)^{\frac{1}{2}-\beta} \\
&\leq C s^{\beta - \frac{1}{2}} (t-s)^{\frac{1}{2}-\beta}\quad\forall\,0 < s < t \leq T.
\end{align*}
\end{theorem}

As an advantage, the power of $t-s$ in Theorem \ref{thm.newRegu} is higher than that in \eqref{eq.xHolder}. Thus, we will adopt Theorem \ref{thm.newRegu} to establish the error analysis of numerical methods in section \ref{sec.NumericalMethods}. The main idea behind the proof of Theorem \ref{thm.newRegu} comes from the convolutional structure of \eqref{eq.equivModel}. Before proving Theorem \ref{thm.newRegu}, we first study the regularity of the convolution $\Lambda_F$ with $F \in \cC([0,T], L^p(\Omega,\hR^d))$. Since $f \circ x \in \cC([0,T], L^p(\Omega,\hR^d))$ in view of Theorem \ref{thm.wellpos}, the following proposition can be used to analyze the regularity of the solution to \eqref{eq.equivModel}.

\begin{proposition} \label{prop.convo}
Let $0 < s < t \leq T$ be arbitrary. For the measurable function $F \in \cC([0,T], L^p(\Omega,\hR^d))$, if there exist $K > 0$, $\mu \in [0,1)$ and $\nu \in [0,1)$ such that
\begin{align} \label{eq.assF}
\|F(t) - F(s)\|_{L^p(\Omega,\hR^d)} \leq K s^{-\mu} (t-s)^{\nu},
\end{align}
then for arbitrarily small $\varepsilon>0$,
\begin{align*}
\| \Lambda_F(t) - \Lambda_F(s) \|_{L^p(\Omega,\hR^d)}
&\leq
\begin{cases}
C ( t^{1-\alpha} - s^{1-\alpha} ) + C s^{-\mu} (t-s)^{ \min\{1-\alpha+\nu,\, 1\} }, &\!\!\!\! \mbox{if } \nu \neq \alpha, \\
C ( t^{1-\alpha} - s^{1-\alpha} ) + C s^{-\mu} (t-s)^{1-\varepsilon}, &\!\!\!\! \mbox{if } \nu = \alpha,
\end{cases}
\end{align*}
where $C>0$ is independent of $t$ and $s$.
\end{proposition}

\begin{proof}
For any $t \in [0,T]$, define $\Phi(t) = \int_0^t (t-u)^{-\alpha} \big( F(u) - F(0) \big) \rd u$, which gives
\begin{align} \label{eq.LambdaF2}
\Lambda_F(t) = \frac{F(0)}{1-\alpha} t^{1-\alpha} + \Phi(t).
\end{align}
For simplicity, let $\tau = t - s$. By the change of variables,
\begin{align*}
\Phi(t) - \Phi(s) &= \int_{-\tau}^{s} (u+\tau)^{-\alpha} \big( F(s-u) - F(0) \big) \rd u - \int_{0}^{s} u^{-\alpha} \big( F(s-u) - F(0) \big) \rd u \\
&= \underbrace{\int_{-\tau}^{s} (u+\tau)^{-\alpha} \big( F(s) - F(0) \big) \rd u - \int_{0}^{s} u^{-\alpha} \big( F(s) - F(0) \big) \rd u}_{=:\, \Psi_1} \\
&\quad + \underbrace{\int_{-\tau}^{0} (u+\tau)^{-\alpha} \big( F(s-u) - F(s) \big) \rd u}_{=:\, \Psi_2} \\
&\quad + \underbrace{\int_{0}^{s} \big( (u+\tau)^{-\alpha} - u^{-\alpha} \big) \big( F(s-u) - F(s) \big) \rd u}_{=:\, \Psi_3}.
\end{align*}
Firstly, $\Psi_1=(1-\alpha)^{-1}(F(s) - F(0)) ( t^{1-\alpha} - s^{1-\alpha} )$ and $F \in \cC([0,T], L^p(\Omega,\hR^d))$ show
\begin{align*}
\| \Psi_1 \|_{L^p(\Omega,\hR^d)} \leq C ( t^{1-\alpha} - s^{1-\alpha} ).
\end{align*}
Secondly, it follows from the Beta function $B(m,n) := \int_{0}^{1} z^{m-1} (1-z)^{n-1} \rd z$ with $m, n > 0$ that for any $-\infty < a < b < \infty$,
\begin{align}\label{eq.BetaInte}
\int_a^b (b-z)^{m-1} (z-a)^{n-1} \rd z
= (b-a)^{m+n-1} B(m,n) \quad \forall\, m, n > 0,
\end{align}
which along with \eqref{eq.assF} indicates
\begin{align*}
\| \Psi_2 \|_{L^p(\Omega,\hR^d)}
\leq C \int_{-\tau}^{0} (u+\tau)^{-\alpha} s^{-\mu} (-u)^{\nu} \rd u 
\leq C s^{-\mu} \tau^{1-\alpha+\nu}.
\end{align*}
Thirdly, \eqref{eq.assF} reveals
\begin{align*}
\| \Psi_3 \|_{L^p(\Omega,\hR^d)}
&\leq C \int_{0}^{s} \int_{0}^{\tau} (u+z)^{-\alpha-1} \rd z \, (s-u)^{-\mu} u^{\nu} \rd u.
\end{align*}
If $\nu \in(\alpha,1)$, then $\int_{0}^{\tau} (u+z)^{-\alpha-1} \rd z \leq C_{\varepsilon} u^{\varepsilon-1-\nu} \tau$. If $\nu =\alpha$, then $\int_{0}^{\tau} (u+z)^{-\alpha-1} \rd z\le u^{\varepsilon-1-\nu} \int_0^\tau(u+z)^{-\varepsilon}\rd z \le C_{\varepsilon}u^{\varepsilon-1-\nu}\tau^{1-\varepsilon}$. Therefore, 
\begin{align*}
\| \Psi_3 \|_{L^p(\Omega,\hR^d)}
\leq
\begin{cases}
C s^{-\mu} \tau, & \mbox{if } \nu \in (\alpha,1), \\
C s^{-\mu} \tau^{1-\varepsilon}, & \mbox{if } \nu = \alpha.
\end{cases}
\end{align*}
When $\nu \in [0, \alpha)$, by \eqref{eq.BetaInte},
\begin{align*}
\| \Psi_3 \|_{L^p(\Omega,\hR^d)}
&\leq C \int_{0}^{s} \int_{0}^{\tau} (u+z)^{\nu-\alpha-1} \rd z \, (s-u)^{-\mu} \rd u \\
&\leq C\left( \int_{0}^{s} (s-u)^{-\mu} u^{\nu-\alpha} \rd u - \int_{0}^{s} (s-u)^{-\mu} (u+\tau)^{\nu-\alpha} \rd u\right) \\
&\leq C \big| s^{1-\mu+\nu-\alpha} - (s+\tau)^{1-\mu+\nu-\alpha} \big| 
 + C \int_{-\tau}^{0} (s-u)^{-\mu} (u+\tau)^{\nu-\alpha} \rd u \\
&\leq C s^{-\mu} \tau^{1-\alpha+\nu}.
\end{align*}
Collecting the estimates of $\{\Psi_i\}_{i=1}^3$ and recalling \eqref{eq.LambdaF2}, we complete the proof.
\end{proof}

\begin{remark} \label{rem.convo}
Under the assumptions of Proposition \ref{prop.convo} with \eqref{eq.assF} replaced by the more general form $\|F(t) - F(s)\|_{L^p(\Omega,\hR^d)} \leq K s^{-\mu} (t-s)^{\nu} + K (t-s)^{q}$, where $q > 0$, one can use a similar proof of Proposition \ref{prop.convo} to obtain
\begin{align*}
\| \Lambda_F(t) - \Lambda_F(s) \|_{L^p(\Omega,\hR^d)}
&\leq C ( t^{1-\alpha} - s^{1-\alpha} ) + C s^{-\mu} (t-s)^{ \min\{1-\alpha+\nu-\varepsilon,\, 1\} } \\
&\quad +
\begin{cases}C (t-s)^{1-\alpha+q}, & \mbox{if } q \neq \alpha, \\
C(t-s)^{1-\varepsilon},& \mbox{if } q = \alpha.
\end{cases}
\end{align*}
\end{remark}

\begin{proof}[Proof of Theorem \ref{thm.newRegu}]
Based on \eqref{eq.assfLip} and \eqref{eq.xHolder}, taking $F = f \circ x$ into Proposition \ref{prop.convo} with $\mu = 0$ and $\nu = \min\{1-\alpha, \frac{1}{2}-\beta\}$ yields
\begin{align} \label{eq.Lambda1}
\left\| \Lambda_{f \circ x}(t) - \Lambda_{f \circ x}(s) \right\|_{L^p(\Omega,\hR^d)} 
\leq C (t^{1-\alpha} - s^{1-\alpha}) + C(t-s)^{\min\{2(1-\alpha),\, \frac{3}{2}-\alpha-\beta,\, 1-\varepsilon\}}.
\end{align}
Then it follows from \eqref{eq.duff} and 
\begin{align*}
t^{1-\alpha} - s^{1-\alpha}
\leq s^{\beta-\frac{1}{2}} t^{\frac{3}{2}-\alpha-\beta} - s^{\beta-\frac{1}{2}} s^{\frac{3}{2}-\alpha-\beta}
\leq C s^{\beta-\frac{1}{2}} (t-s)^{\min\{\frac{3}{2}-\alpha-\beta,\, 1\}}
\end{align*} 
that 
\begin{align} \label{eq.step1}
&\quad\ \| x(t) - x(s) \|_{L^p(\Omega,\hR^d)} \notag\\
&\leq \left\| \Lambda_{f \circ x}(t) - \Lambda_{f \circ x}(s) \right\|_{L^p(\Omega,\hR^d)} + \left\| \Upsilon_{g \circ x}(t) - \Upsilon_{g \circ x}(s) \right\|_{L^p(\Omega,\hR^d)} \notag \\
&\leq C (t^{1-\alpha} - s^{1-\alpha}) + C(t-s)^{\min\{2(1-\alpha),\, \frac{3}{2}-\alpha-\beta,\, 1-\varepsilon\}} + C(t-s)^{\frac{1}{2}-\beta} \notag \\
&\leq C s^{\beta-\frac{1}{2}} (t-s)^{\min\{\frac{3}{2}-\alpha-\beta,\, 1\}} + C(t-s)^{\min\{2(1-\alpha),\, \frac{1}{2}-\beta\}}.
\end{align}
In the estimation of \eqref{eq.Lambda1}, replacing \eqref{eq.xHolder} by \eqref{eq.step1}, and using Remark \ref{rem.convo} with $\mu =\frac{1}{2}- \beta$, $\nu = \min\{\frac{3}{2}-\alpha-\beta, 1\}$ and $q = \min\{2(1-\alpha), \frac{1}{2}-\beta\}$, we have
\begin{align*}
&\quad\, \left\| \Lambda_{f \circ x}(t) - \Lambda_{f \circ x}(s) \right\|_{L^p(\Omega,\hR^d)} \\
&\leq C (t^{1-\alpha} - s^{1-\alpha}) + C s^{\beta-\frac{1}{2}} (t-s)^{\min\{\frac{5}{2}-2\alpha-\beta-\varepsilon,\, 2-\alpha-\varepsilon,\, 1\}} \\
&\quad\, + C (t-s)^{\min\{3(1-\alpha),\, \frac{3}{2}-\alpha-\beta,\, 1-\varepsilon\}} \\
&\leq C s^{\beta-\frac{1}{2}} (t-s)^{\min\{\frac{3}{2}-\alpha-\beta,\, 1\}} + C(t-s)^{\min\{3(1-\alpha),\, \frac{1}{2}-\beta\}}.
\end{align*} 
In this way, we obtain that 
\begin{align} \label{eq.step2} 
\| x(t) - x(s) \|_{L^p(\Omega,\hR^d)} \leq C s^{\beta-\frac{1}{2}} (t-s)^{\min\{\frac{3}{2}-\alpha-\beta,\, 1\}} + C(t-s)^{\min\{3(1-\alpha),\, \frac{1}{2}-\beta\}},
\end{align}
which improves the power ${\min\{2(1-\alpha),\, \frac{1}{2}-\beta\}}$ in \eqref{eq.step1} to ${\min\{3(1-\alpha),\, \frac{1}{2}-\beta\}}$. By repeating these steps from \eqref{eq.step1} to \eqref{eq.step2}, the desired result follows.
\end{proof}

Based on Theorem \ref{thm.newRegu}, one can further use Remark \ref{rem.convo} with $\mu = \beta-\frac{1}{2}$, $\nu = \min\{\frac{3}{2}-\alpha-\beta, 1\}$ and $q = \frac{1}{2}-\beta$ to obtain 
\begin{align} \label{eq.driftSingular}
\left\| \Lambda_{f \circ x}(t) - \Lambda_{f \circ x}(s) \right\|_{L^p(\Omega,\hR^d)}
\leq
\begin{cases}
C s^{\beta-\frac{1}{2}} (t-s)^{\min\{\frac{3}{2}-\alpha-\beta,\, 1\}}, & \mbox{if } \alpha \neq \frac{1}{2}-\beta, \\
C s^{\beta-\frac{1}{2}} (t-s)^{1-\varepsilon}, & \mbox{if } \alpha = \frac{1}{2}-\beta.
\end{cases}
\end{align}

As shown in Theorems \ref{thm.wellpos} and \ref{thm.newRegu}, the solution to \eqref{eq.model} is $(\frac{1}{2}-\beta)$-H\"older continuous on every compact subset of $(0,T]$, while it is only $\min\{1-\alpha, \frac{1}{2}-\beta\}$-H\"older continuous on any neighborhood of $t = 0$. In other words, the solution to \eqref{eq.model} may have an initial singularity in the sense that it exhibits a lower regularity near the initial time $t = 0$ than when $t$ is away from $0$. The initial singularity motivates us to construct the numerical methods for \eqref{eq.model} by using graded meshes in the next section.

\section{Numerical methods}
\label{sec.NumericalMethods}

For the graded mesh $\cM_r$ defined in \eqref{eq:defGM}, denote by $h_n := t_n - t_{n-1}$ $(n= 1, 2, \ldots, N)$ the stepsize. Then
\begin{align} \label{eq.singleSize}
rT N^{-r} (n-1)^{r -1}\leq h_{n}
= T \Big( \big( \frac{n}{N} \big)^r - \big( \frac{n-1}{N} \big)^r \Big)
\leq rT N^{-r} n^{r -1},
\end{align}
which implies
\begin{equation} \label{eq.LaterStepFormerStep}
h_n \leq r 2^{r-1} h_{n-1} \mathbbm{1}_{\{n = 2\}} + 3^{r-1} h_{n-1} \mathbbm{1}_{\{ n \geq 3\}}, 
\end{equation}
since $h_2 \leq r T N^{-r} 2^{r -1} = r 2^{r -1} h_1$ and $h_n \leq rT N^{-r} n^{r -1} \leq 3^{r-1} rT N^{-r} (n-2)^{r -1} \leq 3^{r-1} h_{n-1}$ holds for all $n \geq 3$. For any $k\in\{0,1,\ldots,n-1\}$, 
\begin{align} \label{eq.multiSize}
t_n-t_{k+1}
= T \Big( \big( \frac nN \big)^r - \big( \frac{k+1}{N} \big)^r \Big)
\geq r T N^{-r}(k+1)^{r-1}(n-k-1).
\end{align}
Denote $\hat{s} := t_i$ for $s \in [t_i, t_{i+1})$, $i \in \{0,1,\ldots,N-1\}$. For any $c \in \hR$, it holds that
\begin{align} \label{eq.s}
\hat{s}^c \leq \max\{1,\, 2^{- r c}\} s^c \quad \forall\,s \in [t_1,T).
\end{align}
By \eqref{eq.s} and \eqref{eq.BetaInte}, for any $a, b < 1$,
\begin{align} \label{eq.BetaC}
\int_{t_1}^{t} (t-s)^{-a} \hat{s}^{-b} \rd s 
&\leq \max\{1,\, 2^{r b}\} \int_{t_1}^{t} (t-s)^{-a} s^{-b} \rd s \notag\\
&= \max\{1,\, 2^{r b}\} B(1-a,1-b) (t-t_1)^{1-a-b} \quad \forall\, t \in [t_1,T].
\end{align}

The following discrete Gr\"onwall inequality on graded meshes is helpful for the numerical analysis later, whose detailed proof is provided in Appendix \ref{appendix.Gronwall}.

\begin{lemma} [Gr\"onwall-type inequality] \label{lem.Gronwall}
Let $\mu \in (0,1]$, $\gamma \in (0,1]$, $C_1 > 0$, $C_2 > 0$, $t_n \in \cM_r$ with $n \in \{1,2,\ldots,N\}$ and $r \geq 1$. If the non-negative sequence $\{z_n\}_{n=1}^N$ satisfies
\begin{align*}
z_n \leq C_1 t_n^{\mu-1} + C_2 \sum_{j=1}^{n-1} \int_{t_j}^{t_{j+1}} (t_n-s)^{\gamma-1} z_j \rd s \quad \forall\,n \in \{1,2,\ldots,N\},
\end{align*}
then there exists some constant $C_3 > 0$ independent of $C_1$ such that
\begin{align*}
z_n \leq C_3 C_1t_n^{\mu-1} \quad \forall\,n \in \{1,2,\ldots,N\}.
\end{align*}
\end{lemma}

\subsection{EM method}

For the equation \eqref{eq.model}, the EM method on $\cM_r$ is defined by
\begin{align} \label{eq.EMscheme}
X_{n} = X_{0} + \sum_{i=0}^{n-1} \int_{t_{i}}^{t_{i+1}}
(t_n-s)^{-\alpha} f(X_{i}) \rd s + \sum_{i=0}^{n-1} \int_{t_{i}}^{t_{i+1}} (t_n-t_i)^{-\beta} g(X_{i}) \rd W(s), 
\end{align}
where $n= 1, 2, \ldots, N$ and $X_0 = x_0$. The following theorem establishes its pointwise-in-time error estimate.

\begin{theorem} \label{thm.EM}
Let $\alpha \in (0,1)$, $\beta \in (0,\frac{1}{2})$, and $r \geq 1$. Under the assumptions of Theorem \ref{thm.wellpos}, for any $p \geq 2$, there exists $C > 0$ such that
\begin{align*}
\| x(t_n) - X_n \|_{L^p(\Omega,\hR^d)} \leq C t_n^{1-\alpha+\rho - (\frac{1}{2}-\beta)/r} N^{-(\frac{1}{2}-\beta)} \quad \forall\, n = 1,2,\ldots,N, 
\end{align*}
where $\rho = \min\{1-\alpha, \frac{1}{2}-\beta\}$. 
\end{theorem}

Building on Theorem \ref{thm.EM}, one has the following uniform-in-time error estimates.

\begin{corollary} \label{cor.EMres}
Let the assumptions of Theorem \ref{thm.EM} hold and
\begin{align*}
r^{\,}_{\textup{EM}} := \max\left\{ \frac{\frac{1}{2}-\beta}{1-\alpha+\rho},\, 1 \right\}.
\end{align*}
Then for any $p \geq 2$, there exists $C > 0$ satisfying:\
\begin{itemize}
\item[(1)] when $r = 1$,
\begin{align*} 
 \sup_{1 \leq n \leq N} \| x(t_n) - X_n \|_{L^p(\Omega,\hR^d)} \leq C N^{ - \min\{2(1-\alpha),\, \frac{1}{2}-\beta\} };
 \end{align*}

\item[(2)] when $r = r^{\,}_{\textup{EM}}$,
\begin{align*} 
\sup_{1 \leq n \leq N} \| x(t_n) - X_n \|_{L^p(\Omega,\hR^d)} \leq C N^{-(\frac{1}{2}-\beta)}.
\end{align*}
\end{itemize}
\end{corollary}

In fact, Corollary \ref{cor.EMres}(1) improves the result of order $\min\{ 1-\alpha, \frac{1}{2}-\beta \}$ in the literature. We emphasize that even for the uniform mesh case, the novel regularity estimate (i.e., Theorem \ref{thm.newRegu}) is still helpful for the error analysis of the EM method. Besides, the uniform-in-time convergence order $\frac{1}{2}-\beta$ in Corollary \ref{cor.EMres}(2) is optimal for the EM method of \eqref{eq.model}, in view of \cite[Corollary 2.1]{NualartSaikia2022}.

\begin{proof}[Proof of Theorem \ref{thm.EM}]
We define the continuous-time version of the EM method \eqref{eq.EMscheme} by
\begin{align} \label{eq.cont-EMscheme}
X(t) = X_{0} + \int_{0}^{t} (t-s)^{-\alpha} f(X(\hat{s})) \rd s + \int_{0}^{t} (t-\hat{s})^{-\beta} g(X(\hat{s})) \rd W(s), ~~ t \in [0,T]. 
\end{align} 
Note that $X(t_n) = X_n$ for all $t_n \in \cM_r$. According to \eqref{eq.model} and \eqref{eq.cont-EMscheme}, 
\begin{align*}
x(t_n) - X_n
&= \int_{0}^{t_n} (t_n-s)^{-\alpha} \big( f(x(s)) - f(X(\hat{s})) \big) \rd s \\
&\quad + \int_{0}^{t_n} (t_n-s)^{-\beta} g(x(s)) - (t_n-\hat{s})^{-\beta} g(X(\hat{s})) \rd W(s).
\end{align*}
Then the triangle inequality yields $\| x(t_n) - X_n \|_{L^p(\Omega,\hR^d)} \leq \sum_{l=1}^5\cJ_l$ with
\begin{align*}
\cJ_1 &:= \left\| \int_{0}^{t_n} (t_n-s)^{-\alpha} \big( f(x(s)) - f(x(\hat{s})) \big) \rd s \right\|_{L^p(\Omega,\hR^d)}, \\
\cJ_2 &:= \left\| \int_{0}^{t_n} (t_n-s)^{-\alpha} \big( f(x(\hat{s})) - f(X(\hat{s})) \big) \rd s \right\|_{L^p(\Omega,\hR^d)}, \\
\cJ_3 &:= \left\| \int_{0}^{t_n} (t_n-s)^{-\beta} \big( g(x(s)) - g(x(\hat{s})) \big) \rd W(s) \right\|_{L^p(\Omega,\hR^d)}, \\
\cJ_4 &:= \left\| \int_{0}^{t_n} \big( (t_n-s)^{-\beta} - (t_n-\hat{s})^{-\beta} \big) g(x(\hat{s})) \rd W(s) \right\|_{L^p(\Omega,\hR^d)}, \\
\cJ_5 &:= \left\| \int_{0}^{t_n} (t_n-\hat{s})^{-\beta} \big( g(x(\hat{s})) - g(X(\hat{s})) \big) \rd W(s) \right\|_{L^p(\Omega,\hR^d)}.
\end{align*}

For $\cJ_1$, using \eqref{eq.assfLip}, \eqref{eq.xHolder} and Theorem \ref{thm.newRegu} indicates 
\begin{align*}
\cJ_1
&\leq C \left( \int_{0}^{t_1} + \int_{t_1}^{t_n} \right) (t_n-s)^{-\alpha} \left\| x(s) - x(\hat{s}) \right\|_{L^p(\Omega,\hR^d)} \rd s \\ 
&\leq C \int_{0}^{t_1} (t_n-s)^{-\alpha} s^{\rho} \rd s + C \int_{t_1}^{t_n} (t_n-s)^{-\alpha} \big( \hat{s}^{\beta-\frac{1}{2}} (s-\hat{s})^{\sigma} + (s-\hat{s})^{\frac{1}{2}-\beta} \big) \rd s, 
\end{align*}
where $\rho = \min\{1-\alpha, \frac{1}{2}-\beta\}$ and $\sigma = \min\{\frac{3}{2}-\alpha-\beta, 1\}$. By $(\frac{1}{2}-\beta)-(\rho+1) r < 0$,
\begin{align}
\int_{0}^{t_1} (t_n-s)^{-\alpha} s^{\rho} \rd s
&\leq C t_n^{-\alpha} h_1^{\rho+1} 
\leq C (n/N)^{-\alpha r} N^{-(\rho+1) r} \notag\\
&= C (n/N)^{(1-\alpha+\rho) r -(\frac{1}{2}-\beta)} N^{-(\frac{1}{2}-\beta)} n^{(\frac{1}{2}-\beta)-(\rho+1) r} \notag\\
&\leq C t_n^{1-\alpha+\rho - (\frac{1}{2}-\beta)/r} N^{-(\frac{1}{2}-\beta)}. \label{eq.rho} 
\end{align}
Let $\kappa \in \{\frac{1}{2}-\beta, \sigma\}$. We claim that 
\begin{align} \label{eq.kappa}
 \int_{t_1}^{t_n} (t_n-s)^{-\alpha} \hat{s}^{\beta-\frac{1}{2}} (s-\hat{s})^{\sigma} \rd s
\leq C t_n^{\frac{1}{2}-\alpha+\beta+\sigma - \kappa/r} N^{-\kappa}. 
\end{align}
To prove \eqref{eq.kappa}, we formulate
\begin{align*}
&\quad\ \int_{t_1}^{t_n} (t_n-s)^{-\alpha} \hat{s}^{\beta-\frac{1}{2}} (s-\hat{s})^{\sigma} \rd s \\
&= \left( \sum_{k=1}^{\lceil n/2 \rceil-1} + \sum_{k = \lceil n/2 \rceil}^{n-2} \right) \int_{t_k}^{t_{k+1}} (t_n-s)^{-\alpha} t_k^{\beta-\frac{1}{2}} (s-t_k)^{\sigma} \rd s \\
&\quad + \int_{t_{n-1}}^{t_{n}} (t_n-s)^{-\alpha} t_{n-1}^{\beta-\frac{1}{2}} (s-t_{n-1})^{\sigma} \rd s.
\end{align*}
Then, by \eqref{eq.singleSize} and $r \geq 1$, we have
\begin{align*}
&\quad \sum_{k=1}^{\lceil n/2 \rceil-1} \int_{t_k}^{t_{k+1}} (t_n-s)^{-\alpha} t_k^{\beta-\frac{1}{2}} (s-t_k)^{\sigma} \rd s 
\leq C t_n^{-\alpha} \sum_{k=1}^{\lceil n/2 \rceil-1} t_k^{\beta-\frac{1}{2}} h_{k+1}^{\sigma+1} \\
& \leq C (n/N)^{-\alpha r} \sum_{k=1}^{\lceil n/2 \rceil-1} (k/N)^{(\beta-\frac{1}{2}) r} ( (k+1)^{r -1} N^{-r} )^{\sigma+1} \\
& = C n^{-\alpha r} N^{(\alpha-\beta-\sigma-\frac{1}{2}) r} \sum_{k=1}^{\lceil n/2 \rceil-1} k^{(\beta+\sigma+\frac{1}{2}) r -\sigma-1} 
\leq C t_n^{\frac{1}{2}-\alpha+\beta+\sigma - \kappa/r} N^{-\kappa}
\end{align*}
and 
\begin{align*}
&\quad\, \int_{t_{n-1}}^{t_{n}} (t_n-s)^{-\alpha} t_{n-1}^{\beta-\frac{1}{2}} (s-t_{n-1})^{\sigma} \rd s
\leq C t_{n}^{\beta-\frac{1}{2}} h_{n}^{1-\alpha+\sigma} \\
& \leq C (n/N)^{(\beta-\frac{1}{2}) r} (N^{-r} n^{r -1})^{1-\alpha+\sigma} 
\leq C t_n^{\frac{1}{2}-\alpha+\beta+\sigma - \kappa/r} N^{-\kappa}.
\end{align*} 
In addition, it follows from 
\begin{align*}
&\quad\, n^{-(1-\alpha+\sigma)r+\kappa} \sum_{k = \lceil n/2 \rceil}^{n-2} (k+1)^{-\alpha(r-1)} k^{(\sigma+1)(r -1)} (n-k-1)^{-\alpha} \\
& \leq n^{-(1-\alpha+\sigma)r+\kappa} \sum_{k = \lceil n/2 \rceil}^{n-2}
(k+1)^{(1-\alpha+\sigma)(r-1)} (n-k-1)^{-\alpha} \\
& \leq n^{-(1-\alpha+\sigma)+\kappa} \sum_{k = \lceil n/2 \rceil}^{n-2} (n-k-1)^{-\alpha} 
\leq C,
\end{align*}
\eqref{eq.singleSize} and \eqref{eq.multiSize} that 
\begin{align*}
&\quad\, \sum_{k = \lceil n/2 \rceil}^{n-2} \int_{t_k}^{t_{k+1}} (t_n-s)^{-\alpha} t_k^{\beta-\frac{1}{2}} (s-t_k)^{\sigma} \rd s \\
&\leq C \sum_{k = \lceil n/2 \rceil}^{n-2} \int_{t_k}^{t_{k+1}} (t_n-t_{k+1})^{-\alpha} t_n^{\beta-\frac{1}{2}} (t_{k+1}-t_k)^{\sigma} \rd s \\
&\leq C \sum_{k = \lceil n/2 \rceil}^{n-2} \big( N^{-r} (k+1)^{r-1} (n-k-1) \big)^{-\alpha} \big( n/N \big)^{(\beta-\frac{1}{2}) r} \big( k^{r -1} N^{-r} \big)^{\sigma+1} \\
&= C (n/N)^{(\frac{1}{2}-\alpha+\beta+\sigma) r -\kappa} N^{-\kappa} \\
&\quad \times n^{-(1-\alpha+\sigma)r+\kappa} \sum_{k = \lceil n/2 \rceil}^{n-2} (k+1)^{-\alpha(r-1)} k^{(\sigma+1)(r -1)} (n-k-1)^{-\alpha} \\
&\leq C t_n^{\frac{1}{2}-\alpha+\beta+\sigma - \kappa/r} N^{-\kappa}.
\end{align*} 
Hence \eqref{eq.kappa} holds. Based on \eqref{eq.rho} and \eqref{eq.kappa} with $\kappa = \frac{1}{2}-\beta$, one can read 
\begin{align*}
\cJ_1 
&\leq C t_n^{1-\alpha+\rho - (\frac{1}{2}-\beta)/r} N^{-(\frac{1}{2}-\beta)} + C t_n^{\frac{1}{2}-\alpha+\beta+\sigma - (\frac{1}{2}-\beta)/r} N^{-(\frac{1}{2}-\beta)} + C t_n^{1-\alpha} N^{-(\frac{1}{2}-\beta)} \\
&\leq C t_n^{1-\alpha+\rho - (\frac{1}{2}-\beta)/r} N^{-(\frac{1}{2}-\beta)}, 
\end{align*}
since $\sigma \geq \rho+\frac{1}{2}-\beta$. 

For $\cJ_2$ and $\cJ_5$, it follows from the H\"older inequality, Burkholder--Davis--Gundy (BDG) inequality and \eqref{eq.assfLip} that 
\begin{align}\label{eq:J2J5}
|\cJ_2|^2+ |\cJ_5|^2 \leq C \int_{0}^{t_n} (t_n-s)^{- \max\{ \alpha, 2\beta \}} \| x(\hat{s}) - X(\hat{s}) \|_{L^p(\Omega,\hR^d)}^2 \rd s. 
\end{align}
For $\cJ_3$, by the BDG inequality, \eqref{eq.assfLip}, \eqref{eq.xHolder}, Theorem \ref{thm.newRegu} and \eqref{eq.s},
\begin{align*}
|\cJ_3|^2 
&\leq C \int_{0}^{t_n} (t_n-s)^{-2\beta} \| x(s) - x(\hat{s}) \|_{L^p(\Omega,\hR^d)}^2 \rd s \\
& \leq C \int_{0}^{t_1} (t_1-s)^{-2\beta} s^{2\rho} \rd s+ C \int_{t_1}^{t_n} (t_n-s)^{-2\beta}
s^{2\beta-1} (s-\hat{s})^{1-2\beta} \rd s \\
&\leq C N^{-(1-2\beta)}. 
\end{align*} 
In virtue of \eqref{eq.singleSize}, we have $h_{n} \leq CN^{-1}$, and then 
\begin{align}\label{eq.diffSqureSingle}
&\quad\, \int_{0}^{t_n} \big| (t_n-s)^{-\beta} - (t_n-\hat{s})^{-\beta} \big|^2 \rd s \\
&= \int_0^{t_{n-1}} \big| (t_n-s)^{-\beta} - (t_n-\hat{s})^{-\beta} \big|^2 \rd s + \int_{t_{n-1}}^{t_n} \big| (t_n-s)^{-\beta} - (t_n-\hat{s})^{-\beta} \big|^2 \rd s \notag\\
&\leq C \int_0^{t_{n-1}} \Big| \int_{\hat{s}}^s (t_n-u)^{-(1+\beta)} \rd u \Big|^2 \rd s + \int_{t_{n-1}}^{t_n} (t_n-s)^{-2\beta} \rd s \notag\\
&\leq C \int_0^{t_{n-1}} (t_n-s)^{-(2+\beta)} \Big| \int_{\hat{s}}^s (t_n-u)^{-\frac{\beta}{2}} \rd u \Big|^2 \rd s + C h_{n}^{1-2\beta} \notag\\
&\leq C h_n^{2-\beta} \int_0^{t_{n-1}} (t_n-s)^{-(2+\beta)} \rd s + C N^{-(1-2\beta)} \leq C N^{-(1-2\beta)}. \notag
\end{align}
The BDG inequality, Theorem \ref{thm.wellpos} and \eqref{eq.diffSqureSingle} imply $|\cJ_4|^2 \leq C N^{-(1-2\beta)}$. Collecting the above estimates for $\{\cJ_i\}_{i=1}^5$ and using the Cauchy--Schwarz inequality read 
\begin{align*}
\| x(t_n) - X(t_n) \|_{L^p(\Omega,\hR^d)}^2
&\leq C \left( t_n^{1-\alpha+\rho - (\frac{1}{2}-\beta)/r} N^{-(\frac{1}{2}-\beta)} \right)^2 \\&\quad+ C \int_{0}^{t_n} (t_n-s)^{-\max\{ \alpha, 2\beta \}} \left\| x(\hat{s}) - X(\hat{s}) \right\|_{L^p(\Omega,\hR^d)}^2 \rd s. \end{align*}
Finally, applying the Gr\"onwall inequality (i.e., Lemma \ref{lem.Gronwall}) completes the proof.
\end{proof}

\subsection{Fast EM method}

Due to the memory of kernels, the implementation of a single path of the EM method \eqref{eq.EMscheme} needs a computational cost of $\cO\big( N^2 \big)$. To reduce this cost, we next introduce a fast variant of the EM method, whose construction will use the following sum-of-exponentials approximation presented in \cite{JiangZhang2017}.

\begin{lemma} \label{lem.SOE}
Let $\gamma \in (0,1)$, tolerance error $\epsilon \ll 1$, and truncation $\delta \in (0,T]$ be arbitrary. Then there exist positive numbers $\{ \tau_k \}_{k=1}^{N_{\exp}}$ and $\{ \omega_k \}_{k=1}^{N_{\exp}}$ such that 
\begin{align*}
\bigg| t^{-\gamma} - \sum_{k=1}^{N_{\exp}} \omega_k e^{-\tau_k t} \bigg| \leq \epsilon \quad \forall\, t\in[\delta,T],
\end{align*}
where $N_{\exp}$ is a positive integer satisfying
\begin{align*}
N_{\exp} = \cO\Big( \log\frac{1}{\epsilon}\big( \log\log\frac{1}{\epsilon} + \log\frac{T}{\delta} \big) + \log\frac{1}{\delta} \big(\log\log\frac{1}{\epsilon} + \log\frac{1}{\delta}\big) \Big).
\end{align*}
\end{lemma}

This lemma tells us that the kernel function $t^{-\gamma}$ (provided away from the singularity point zero) can be approximated by a summation of exponential functions under a specified precision $\epsilon \ll 1$. Then the EM method \eqref{eq.EMscheme} can be modified into
\begin{align*}
\widetilde{X}_{n} 
&= \widetilde{X}_{0} + \sum_{i=0}^{n-2} \int_{t_{i}}^{t_{i+1}} \sum_{k=1}^{N_{\exp,\alpha}} \omega_{k,\alpha} e^{-\tau_{k,\alpha} (t_n-s)} f(\widetilde{X}_{i}) \rd s + \int_{t_{n-1}}^{t_{n}} (t_n-s)^{-\alpha} f(\widetilde{X}_{n-1}) \rd s \notag \\
&\quad + \sum_{i=0}^{n-2} \int_{t_{i}}^{t_{i+1}} \sum_{k=1}^{N_{\exp,\beta}} \omega_{k,\beta} e^{-\tau_{k,\beta} (t_n-t_i)} g(\widetilde{X}_{i}) \rd W(s) \\
&\quad + \int_{t_{n-1}}^{t_{n}} (t_n-t_{n-1})^{-\beta} g(\widetilde{X}_{n-1}) \rd W(s)
\end{align*}
for $n= 2, 3, \ldots, N$ with $\widetilde{X}_0 = X_0 = x_0$ and $\widetilde{X}_1 = X_1$. Here, $\{\tau_{k,\alpha}, \omega_{k,\alpha}\}_{k=1}^{N_{\exp,\alpha}}$ and $\{\tau_{k,\beta}, \omega_{k,\beta}\}_{k=1}^{N_{\exp,\beta}}$ correspond to the parameters of the sum-of-exponentials approximations of $(t_n-s)^{-\alpha}$ (where $s \in [0,t_{n-1}]$) and $(t_n-t_i)^{-\beta}$ (where $0\leq i \leq n-2$), respectively. Exchanging the summations order and denoting
\begin{align*}
\zeta_{k,\alpha}^n
& :=
\begin{cases}
0, &\mbox{\ if } n = 1, \\
\sum_{i=0}^{n-2} \int_{t_i}^{t_{i+1}} e^{-\tau_{k,\alpha} (t_n-s)} f(\widetilde{X}_{i}) \rd s, &\mbox{\ if } n \geq 2, \\
\end{cases} \\
\zeta_{k,\beta}^n
& :=
\begin{cases}
0, &\mbox{\ if } n = 1, \\
\sum_{i=0}^{n-2} \int_{t_i}^{t_{i+1}} e^{-\tau_{k,\beta} (t_n-t_i)} g(\widetilde{X}_{i}) \rd W(s), &\mbox{\ if } n \geq 2, 
\end{cases} 
\end{align*}
we can construct the fast EM method:\ for $n= 1, 2, \ldots, N$, 
\begin{align} \label{eq.fastEMscheme}
\widetilde{X}_{n} 
&= \widetilde{X}_{0} + \sum_{k=1}^{N_{\exp,\alpha}} \omega_{k,\alpha} \zeta_{k,\alpha}^n + \int_{t_{n-1}}^{t_{n}} (t_n-s)^{-\alpha} f(\widetilde{X}_{n-1}) \rd s \notag\\
&\quad + \sum_{k=1}^{N_{\exp,\beta}} \omega_{k,\beta} \zeta_{k,\beta}^n + \int_{t_{n-1}}^{t_{n}} (t_n-t_{n-1})^{-\beta} g(\widetilde{X}_{n-1}) \rd W(s). 
\end{align}

\begin{remark} 
For the fast EM method \eqref{eq.fastEMscheme} on the graded mesh $\cM_r$, if we take $\delta = h_1$ and $\epsilon = N^{-(\frac{1}{2}-\beta)}$, then Lemma \ref{lem.SOE} implies $N_{\exp,\alpha} = \cO\big( (\log N)^2 \big)$ and $N_{\exp,\beta} = \cO\big( (\log N)^2 \big)$. On the other hand, the semi-group property of exponential functions shows that for all $n \in \{2,3,\ldots,N\}$, the recurrence relations
\begin{align*}
\zeta_{k,\alpha}^n &= e^{-\tau_{k,\alpha}h_n} \zeta_{k,\alpha}^{n-1} + \int_{t_{n-2}}^{t_{n-1}} e^{-\tau_{k,\alpha} (t_n-s)} f(\widetilde{X}_{n-2}) \rd s, \\
\zeta_{k,\beta}^n &= e^{-\tau_{k,\beta}h_n} \zeta_{k,\beta}^{n-1} + \int_{t_{n-2}}^{t_{n-1}} e^{-\tau_{k,\beta} (t_n-t_{n-2})} g(\widetilde{X}_{n-2}) \rd W(s)
\end{align*}
hold. Thus, the implementation of a single sample of the fast EM method requires a computational cost of $\cO\big( N (\log N)^2 \big)$, which reduces that of the EM method. 
\end{remark}

Based on Theorem \ref{thm.EM}, similar to the proof of Theorem 4 of \cite{DaiXiao2020}, one can obtain the following error estimate for the fast EM method \eqref{eq.fastEMscheme}.

\begin{theorem} \label{thm.fastEM}
Under the assumptions of Theorem \ref{thm.EM}, there exists $C > 0$ independent of $N$ and $\epsilon$ such that for all $n \in \{1,2,\ldots,N\}$,
\begin{align*}
\| x(t_n) - \widetilde{X}_n \|_{L^p(\Omega,\hR^d)} \leq C t_n^{1-\alpha+\rho - (\frac{1}{2}-\beta)/r} N^{-(\frac{1}{2}-\beta)} + C\epsilon.
\end{align*}
\end{theorem}
When taking $\epsilon = N^{-(\frac{1}{2}-\beta)}$, the results in Corollary \ref{cor.EMres} for the EM method \eqref{eq.EMscheme} can be extended to the fast EM method \eqref{eq.fastEMscheme}.

\subsection{Milstein method}

For higher computational accuracy, we now study the Milstein method
\begin{align} \label{eq.Milstein}
&Y_n = Y_0 + \sum_{i=0}^{n-1} \int_{t_{i}}^{t_{i+1}} (t_n-s)^{-\alpha} f(Y_i) \rd s + \sum_{i=0}^{n-1} \int_{t_{i}}^{t_{i+1}} (t_n-s)^{-\beta} g(Y_i) \rd W(s) \\
&\ + \sum_{i=0}^{n-1} \int_{t_{i}}^{t_{i+1}} (t_n-s)^{-\beta} g'(Y_i) \sum_{j=0}^{i-1} \int_{t_j}^{t_{j+1}} \big( (s-u)^{-\beta} - (t_i-u)^{-\beta} \big) g(Y_j) \rd W(u) \rd W(s) \notag\\
&\ + \sum_{i=0}^{n-1} \int_{t_{i}}^{t_{i+1}} (t_n-s)^{-\beta} g'(Y_i) \int_{t_i}^s (s-u)^{-\beta} g(Y_i) \rd W(u) \rd W(s), \notag
\end{align}
where $n= 1, 2, \ldots, N$ and $Y_0 = x_0$. Notice that the Milstein method \eqref{eq.Milstein} is easier to implement than those in the literature, since it does not include the terms
\begin{align*}
\sum_{i=0}^{n-1} \int_{t_{i}}^{t_{i+1}} (t_n-s)^{-\alpha} f'(Y_i) \bigg( \sum_{j=0}^{i-1} \int_{t_j}^{t_{j+1}} \big( (s-u)^{-\beta} - (t_i-u)^{-\beta} \big) g(Y_j) \rd W(u) \bigg) \rd s
\end{align*}
and
\begin{align*}
\sum_{i=0}^{n-1} \int_{t_{i}}^{t_{i+1}} (t_n-s)^{-\alpha} f'(Y_i) \bigg( \int_{t_i}^s (s-u)^{-\beta} g(Y_i) \rd W(u) \bigg) \rd s.
\end{align*}
In the error analysis, it also enables us to weaken the condition $f \in \cC_b^2(\hR^d, \hR^d)$ used in the literature. The following theorem presents the pointwise-in-time error estimate for the Milstein method \eqref{eq.Milstein}.

\begin{theorem} \label{thm.Milstein}
Let $\alpha \in (0,1)$, $\beta \in (0,\frac{1}{2})$, and $r \geq 1$. Assume that $f \in \cC_b^1(\hR^d, \hR^d)$ and $g \in \cC_b^2(\hR^d, \hR^{d \times m})$. Then for any $p \geq 2$, there exists $C > 0$ such that 
\begin{align*}
\| x(t_n) - Y_n \|_{L^p(\Omega,\hR^d)}
&\leq C t_n^{ \frac{1}{2}-\alpha+\beta+\sigma - \sigma/r } N^{ -\min\{\frac{3}{2}-\alpha-\beta,\, 1-2\beta\} } \quad \forall\, n = 1,2,\ldots,N, 
\end{align*}
where $\sigma = \min\{\frac{3}{2}-\alpha-\beta, 1\}$.
\end{theorem}

The following uniform-in-time error estimates are consequences of Theorem \ref{thm.Milstein}.

\begin{corollary} \label{cor.Milsteinres}
Let the assumptions of Theorem \ref{thm.Milstein} hold and
\begin{align*}
r^{\,}_{\textup{Milstein}} := \max\bigg\{ \frac{\sigma}{\frac{1}{2}-\alpha+\beta+\sigma},\, 1 \bigg\}.
\end{align*}
Then for any $p \geq 2$, there exists $C > 0$ satisfying:\
\begin{itemize}
\item[(1)] when $r = 1$,
\begin{align*} 
\sup_{1 \leq n \leq N} \| x(t_n) - Y_n \|_{L^p(\Omega,\hR^d)} \leq C N^{ -\min\{2(1-\alpha),\, 1-2\beta \} };
\end{align*}

\item[(2)] when $r = r^{\,}_{\textup{Milstein}}$,
\begin{align*} 
\sup_{1 \leq n \leq N} \| x(t_n) - Y_n \|_{L^p(\Omega,\hR^d)} \leq C N^{ -\min\{\frac{3}{2}-\alpha-\beta,\, 1-2\beta\} }.
\end{align*}
\end{itemize}
\end{corollary}

For the Milstein method of \eqref{eq.model}, the convergence order in Corollary \ref{cor.Milsteinres}(2) equals to the convergence order $\frac{1}{2}-\beta$ of the EM method plus the H\"older continuous exponent $\min\{1-\alpha, \frac{1}{2}-\beta\}$ of the exact solution on $[0, T]$, which is consistent with the corresponding results on the numerical study of stochastic differential equations (see e.g., \cite{KloedenPlaten1992, MilsteinTretyakov2004}).

For any $s \in [t_i,t_{i+1})$ with $i \in \{0,1,\ldots, N-1\}$, Taylor's formula tells us that
\begin{align*}
f(x(s)) = f(x(t_i)) + \Delta_f^1(s), \quad
g(x(s)) = g(x(t_i)) + g'(x(t_i)) (x(s) - x(t_i)) + \Delta_g^2(s).
\end{align*}
Here, $\Delta_f^1(s) = \int_0^1 f'(\xi_s^{\theta}) \big(x(s) - x(t_i)\big) \mathrm d \theta$ is the first-order remainder term with $\xi_s^{\theta} = x(t_i) + \theta \big(x(s) - x(t_i)\big)$, and $\Delta_g^2(s)$ denotes the second-order remainder term and satisfies
\begin{align}\label{eq.fg}
\|\Delta_g^2(s)\|_{L^p(\Omega,\mathbb{R}^d)} &\leq C\|x(s) - x(t_i)\|_{L^{2p}(\Omega,\mathbb{R}^d)}^2 
\end{align}
provided that $g \in \cC_b^2(\hR^d, \hR^{d \times m})$. Recalling \eqref{eq.equivModel}, one can obtain the expansion
\begin{align} \label{eq.MilsteinExpan}
&x(t_n) = x_0 + \sum_{i=0}^{n-1} \int_{t_{i}}^{t_{i+1}} (t_n-s)^{-\alpha} f(x(t_i)) \rd s + \sum_{i=0}^{n-1} \int_{t_{i}}^{t_{i+1}} (t_n-s)^{-\beta} g(x(t_i)) \rd W(s) \\
&\ + \sum_{i=0}^{n-1} \int_{t_{i}}^{t_{i+1}} (t_n-s)^{-\beta} g'(x(t_i)) \int_0^{t_i} \big( (s-u)^{-\beta} - (t_i-u)^{-\beta} \big) g(x(u)) \rd W(u) \rd W(s) \notag\\
&\ + \sum_{i=0}^{n-1} \int_{t_{i}}^{t_{i+1}} (t_n-s)^{-\beta} g'(x(t_i)) \int_{t_i}^s (s-u)^{-\beta} g(x(u)) \rd W(u) \rd W(s) +\sum_{j=1}^3 \cA_j(t_n), \notag
\end{align}
where 
\begin{align*}
\cA_1(t_n) &:= \sum_{i=0}^{n-1} \int_{t_{i}}^{t_{i+1}} (t_n-s)^{-\alpha} \Delta_f^1(s) \rd s, \\
\cA_2(t_n) &:= \sum_{i=0}^{n-1} \int_{t_{i}}^{t_{i+1}} (t_n-s)^{-\beta} g'(x(t_i)) \left( \Lambda_{f \circ x}(s) - \Lambda_{f \circ x}(t_i) \right) \rd W(s), \\
\cA_3(t_n) &:= \sum_{i=0}^{n-1} \int_{t_{i}}^{t_{i+1}} (t_n-s)^{-\beta} \Delta_g^2(s) \rd W(s).
\end{align*}
Then, by discarding the remainder terms $\{\cA_j\}_{j=1}^3$ and approximating $g(x(\cdot))$ in \eqref{eq.MilsteinExpan} based on the left rectangle rule, we can construct the Milstein method \eqref{eq.Milstein}, which has the following continuous-time version 
\begin{align} \label{eq.cont-Milstein}
Y(t) &= Y_0 + \int_{0}^{t} (t-s)^{-\alpha} f(Y(\hat{s})) \rd s + \int_{0}^{t} (t-s)^{-\beta} g(Y(\hat{s})) \rd W(s) \\
&\quad + \int_{0}^{t} (t-s)^{-\beta} g'(Y(\hat{s})) \int_{0}^{\hat{s}} \big( (s-u)^{-\beta} - (\hat{s}-u)^{-\beta} \big) g(Y(\hat{u})) \rd W(u) \rd W(s) \notag\\
&\quad + \int_{0}^{t} (t-s)^{-\beta} g'(Y(\hat{s})) \int_{\hat{s}}^s (s-u)^{-\beta} g(Y(\hat{u})) \rd W(u) \rd W(s), \quad t \in[0,T]. \notag
\end{align}
Note that $Y(t_n) = Y_n$ for all $t_n \in \cM_r$.

We prepare several lemmas for proving Theorem \ref{thm.Milstein}. 

\begin{lemma} \label{lem.Ak}
Under the assumptions of Theorem \ref{thm.Milstein},
\begin{align}\label{eq:Aj}
\big\| \cA_j(t_n) \big\|_{L^p(\Omega,\hR^d)}^2
\leq C t_n^{2(\frac{1}{2}-\alpha+\beta+\sigma - \sigma/r)} N^{ -2\min\{\frac{3}{2}-\alpha-\beta,\, 1-2\beta\} }, \quad j=1,2,3,
\end{align}
where $\sigma = \min\{\frac{3}{2}-\alpha-\beta, 1\}$ and $\cA_j$ ($j=1,2,3$) are the same as those in \eqref{eq.MilsteinExpan}.
\end{lemma}

\begin{proof} 
For $\cA_1(t_n)$, we divided it into two terms
\begin{align*}
\cA_{1,1}(t_n):=\sum_{i=0}^{n-1} \int_{t_{i}}^{t_{i+1}} (t_n-s)^{-\alpha} \int_0^1 f'(\xi_s^{\theta}) \big( \Lambda_{f \circ x}(s) - \Lambda_{f \circ x}(t_i) \big) \mathrm d \theta \rd s, \\
\cA_{1,2}(t_n):=\sum_{i=0}^{n-1} \int_{t_{i}}^{t_{i+1}} (t_n-s)^{-\alpha} \int_0^1 f'(\xi_s^{\theta}) \big( \Upsilon_{g \circ x}(s) - \Upsilon_{g \circ x}(t_i) \big) \mathrm d \theta \rd s. 
\end{align*}
Since $\sigma-(2-\alpha) r < 0$ and $\sigma\le \frac32-\alpha-\beta$,
\begin{align}\label{eq.sigma3}
\int_{0}^{t_{1}} (t_n-s)^{-\alpha} (s-\hat{s})^{1-\alpha} \rd s
&\leq C t_n^{-\alpha} h_1^{2-\alpha} \leq C (n/N)^{-\alpha r} N^{-(2-\alpha) r} \notag \\
&= C (n/N)^{(2-2\alpha) r - \sigma} N^{-\sigma} n^{\sigma-(2-\alpha) r} \notag \\
& \leq C t_n^{2-2\alpha - \sigma/r} N^{-\sigma}\le C t_n^{\frac{1}{2}-\alpha+\beta+\sigma - \sigma/r} N^{-\sigma}.
\end{align}
Hence, by the boundedness of $f'$, \eqref{eq.drift} and \eqref{eq.sigma3},
\begin{align*}
\left\| \cA_{1,1}(t_n) \right\|_{L^p(\Omega,\hR^d)}
&\leq Ct_n^{2-2\alpha - \sigma/r} N^{-\sigma} \\
&\quad + C \int_{t_1}^{t_{n}} (t_n-s)^{-\alpha} \left\| \Lambda_{f \circ x}(s) - \Lambda_{f \circ x}(\hat{s}) \right\|_{L^p(\Omega,\hR^d)} \rd s.
\end{align*}
Then according to \eqref{eq.driftSingular} and \eqref{eq.kappa} with $\kappa = \sigma$, the second term on the right hand side can be bounded by $C t_n^{\frac{1}{2}-\alpha+\beta+\sigma - \sigma/r} N^{-\sigma} $ if $\alpha \neq \frac{1}{2}-\beta$, and by $CN^{-(1-2\beta)}$ if $\alpha = \frac{1}{2}-\beta$. These yield 
\begin{align*}
\left\| \cA_{1,1}(t_n) \right\|_{L^p(\Omega,\hR^d)}\leq C t_n^{\frac{1}{2}-\alpha+\beta+\sigma - \sigma/r} N^{ -\min\{\frac{3}{2}-\alpha-\beta,\, 1-2\beta\} }.
\end{align*}

By the stochastic Fubini theorem, $\cA_{1,2}(t_n)=\cA_{1,2}^{\star} (t_n) +\cA_{1,2}^{\star\star} (t_n) $ with 
\begin{align*}
\cA_{1,2}^{\star} (t_n) &:= \sum_{i=0}^{n-1} \int_{t_{i}}^{t_{i+1}} (t_n-s)^{-\alpha} f'(\xi_s^{\theta}) \int_0^{t_i} \big( (s-u)^{-\beta} - (t_i-u)^{-\beta} \big) g(x(u)) \rd W(u) \rd s \\
&\phantom{:}=\int_{0}^{t_{n-1}} \left( \int_{\check{u}}^{t_{n}} (t_n-s)^{-\alpha} f'(\xi_s^{\theta}) \big( (s-u)^{-\beta} - (\hat{s}-u)^{-\beta} \big) \rd s \right) g(x(u)) \rd W(u),\\
\cA_{1,2}^{\star\star} (t_n)&:= \sum_{i=0}^{n-1} \int_{t_{i}}^{t_{i+1}} (t_n-s)^{-\alpha} f'(\xi_s^{\theta}) \int_{t_i}^s (s-u)^{-\beta} g(x(u)) \rd W(u) \rd s \\
&\phantom{:}=\int_{0}^{t_{n}} \left( \int_u^{\check{u}} (t_n-s)^{-\alpha} f'(\xi_s^{\theta}) (s-u)^{-\beta} \rd s \right) g(x(u)) \rd W(u). 
\end{align*}
The BDG inequality and \eqref{lem.inA2} show 
\begin{align*}
\left\| \cA_{1,2}^{\star} (t_n) \right\|_{L^p(\Omega,\hR^d)}^2 
\leq C N^{-2\min\{\frac{3}{2}-\alpha-\beta,\, 1-2\beta\}}.
\end{align*}
Analogously, the BDG inequality also gives
\begin{align*}
&\quad\, \left\| \cA_{1,2}^{\star\star} (t_n) \right\|_{L^p(\Omega,\hR^d)}^2 \\
&\leq C \int_{0}^{t_{n-1}} \left| \int_u^{\check{u}} (t_n-s)^{-\alpha} (s-u)^{-\beta} \rd s \right|^2\!\! \rd u
+ C \int_{t_{n-1}}^{t_n} \left| \int_u^{t_n} (t_n-s)^{-\alpha} (s-u)^{-\beta} \rd s \right|^2\!\! \rd u \\
&\leq C \int_{0}^{t_{n-1}} (t_n-u)^{-2\alpha} \left| \int_u^{\check{u}} (s-u)^{-\beta} \rd s \right|^2 \!\! \rd u
+ C \int_{t_{n-1}}^{t_n} (t_n - u)^{2(1-\alpha-\beta)} \rd u \\
&\leq C N^{-2\min\{\frac{3}{2}-\alpha-\beta,\, 1-2\beta\}}.
\end{align*}
Combining the above estimates completes the proof of \eqref{eq:Aj} for $j=1$.

To estimate $\cA_2(t_n)$, note that 
\begin{align*}
\int_0^{t_1} (t_n-s)^{-2\beta} \| \Lambda_{f \circ x}(s) - \Lambda_{f \circ x}(\hat{s}) \|_{L^p(\Omega,\hR^d)}^2 \rd s 
&\le \int_{0}^{t_1} (t_n-s)^{-2\beta} s^{2(1-\alpha)} \rd s \\ 
&\le CN^{-2(\frac{3}{2}-\alpha-\beta)}.
\end{align*}
In addition, in view of \eqref{eq.driftSingular}, one has that for $s\in(t_1,T]$, 
\begin{align*}
\| \Lambda_{f \circ x}(s) - \Lambda_{f \circ x}(\hat{s}) \|_{L^p(\Omega,\hR^d)}\le C \hat{s}^{\beta-\frac{1}{2}}(s-\hat{s})^{\min\{\frac{3}{2}-\alpha-\beta, 1-2\beta\}}. 
\end{align*}
Hence the BDG inequality, the boundedness of $g'$, \eqref{eq.drift} and \eqref{eq.driftSingular} yield 
\begin{align*}
\left\| \cA_2(t_n) \right\|_{L^p(\Omega,\hR^d)}^2
&\leq C \int_{0}^{t_{n}} (t_n-s)^{-2\beta} \left\| \Lambda_{f \circ x}(s) - \Lambda_{f \circ x}(\hat{s}) \right\|_{L^p(\Omega,\hR^d)}^2 \rd s \\
&\leq CN^{-2(\frac{3}{2}-\alpha-\beta)}\!+\!C \int_{t_1}^{t_{n}} \! \!(t_n\!-\!s)^{-2\beta} \hat{s}^{2\beta-1}(s\!-\!\hat{s})^{2\min\{\frac{3}{2}-\alpha-\beta, 1-2\beta\}} \rd s \\
&\leq C N^{-2\min\{\frac{3}{2}-\alpha-\beta, 1-2\beta\}}. 
\end{align*}
Thus \eqref{eq:Aj} holds for $j=2$.

For $ \cA_3(t_n) $, using the BDG inequality and \eqref{eq.fg} shows
\begin{align*}
\left\| \cA_3(t_n) \right\|_{L^p(\Omega,\hR^d)}^2
\leq C \int_{0}^{t_{n}} (t_n-s)^{-2\beta} \left\| x(s) - x(\hat{s}) \right\|_{L^{2p}(\Omega,\hR^d)}^4 \rd s,
\end{align*}
which together with Theorems \ref{thm.wellpos} and \ref{thm.newRegu} implies 
\begin{align*}
\left\| \cA_3(t_n) \right\|_{L^p(\Omega,\hR^d)}^2
&\leq C N^{-4\min\{1-\alpha,\, \frac{1}{2}-\beta\}} \int_{0}^{t_{1}} (t_n-s)^{-2\beta} \rd s \\
&\quad + C N^{-2\min\{1-\alpha,\, \frac{1}{2}-\beta\}} \int_{t_1}^{t_{n}} (t_n-s)^{-2\beta} \hat{s}^{2\beta-1} (s-\hat{s})^{1-2\beta} \rd s \\
&\leq C N^{-2\min\{\frac{3}{2}-\alpha-\beta,\, 1-2\beta\}}.
\end{align*}
Thus \eqref{eq:Aj} holds for $j=3$. The proof is completed. 
\end{proof}

\begin{lemma} \label{lem.exMilst}
Under the assumptions of Theorem \ref{thm.Milstein}, 
\begin{align*}
\| x(t_n) - Y(t_n) \|_{L^p(\Omega,\hR^d)} \leq C t_n^{1-\alpha+\rho - (\frac{1}{2}-\beta)/r} N^{-(\frac{1}{2}-\beta)}\quad\forall\,n \in \{1,2,\ldots, N\}.
\end{align*}
\end{lemma}

\begin{proof}
It can be verified that 
$\| X(t_n) - Y(t_n) \|_{L^p(\Omega,\hR^d)} \leq C N^{-(\frac{1}{2}-\beta)}$ for any $n \in \{1,2,\ldots, N\}$, 
where $X$ and $Y$ denote the numerical solutions of the EM method \eqref{eq.cont-EMscheme} and the Milstein method \eqref{eq.cont-Milstein}, respectively. Then the desired result follows from Theorem \ref{thm.EM} and the triangle inequality.
\end{proof}

\begin{proof}[Proof of Theorem \ref{thm.Milstein}]
According to \eqref{eq.MilsteinExpan} and \eqref{eq.cont-Milstein}, one can obtain
\begin{align*}
&\quad\, x(t_n) - Y(t_n) - \big( \cA_1(t_n) + \cA_2(t_n) + \cA_3(t_n) \big) \notag\\
&= \int_{0}^{t_n} (t_n-s)^{-\alpha} \big( f(x(\hat{s})) - f(Y(\hat{s})) \big) \rd s + \int_{0}^{t_n} (t_n-s)^{-\beta} \big( g(x(\hat{s})) - g(Y(\hat{s})) \big) \rd W(s) \notag\\
&\quad + \bigg( \int_{0}^{t_n} (t_n-s)^{-\beta} g'(x(\hat{s})) \int_0^{\hat{s}} \big( (s-u)^{-\beta} - (\hat{s}-u)^{-\beta} \big) g(x(u)) \rd W(u) \rd W(s) \notag\\
&\qquad\quad - \int_{0}^{t_n} (t_n-s)^{-\beta} g'(Y(\hat{s})) \int_0^{\hat{s}} \big( (s-u)^{-\beta} - (\hat{s}-u)^{-\beta} \big) g(Y(\hat{u})) \rd W(u) \rd W(s) \bigg) \notag\\
&\quad + \bigg( \int_{0}^{t_n} (t_n-s)^{-\beta} g'(x(\hat{s})) \int_{\hat{s}}^s (s-u)^{-\beta} g(x(u)) \rd W(u) \rd W(s) \notag\\
&\qquad\quad - \int_{0}^{t_n} (t_n-s)^{-\beta} g'(Y(\hat{s})) \int_{\hat{s}}^s (s-u)^{-\beta} g(Y(\hat{u})) \rd W(u) \rd W(s) \bigg) \notag\\
&=: \cB_1(t_n) + \cB_2(t_n) + \cB_3(t_n) + \cB_4(t_n),
\end{align*}
which implies
\begin{align*} 
\big\| x(t_n) - Y(t_n) \big\|_{L^p(\Omega,\hR^d)}^2 \leq 7 \sum_{j=1}^3 \big\| \cA_j(t_n) \big\|_{L^p(\Omega,\hR^d)}^2 + 7 \sum_{k=1}^4 \big\| \cB_k(t_n) \big\|_{L^p(\Omega,\hR^d)}^2.
\end{align*}
Here, $\{\cA_j(t_n)\}_{j=1}^3$ has been bounded in Lemma \ref{lem.Ak}. Similarly to \eqref{eq:J2J5}, we have 
\begin{align*} 
\sum_{k=1}^2\big\| \cB_k(t_n) \big\|_{L^p(\Omega,\hR^d)}^2
\leq C \int_{0}^{t_n} (t_n-s)^{-\max\{\alpha,2\beta\}} \big\| x(\hat{s}) - Y(\hat{s}) \big\|_{L^p(\Omega,\hR^d)}^2 \rd s.
\end{align*}
It remains to estimate $\cB_3(t_n)$ and $\cB_4(t_n)$. 

\textbf{Estimate of $\cB_3(t_n)$:} we rewrite $\cB_3(t_n)=\sum_{l=1}^3\cB_{3,l}(t_n)$ with 
{\footnotesize
\begin{align*}
&\cB_{3,1}(t_n) :=\int_{0}^{t_n} (t_n-s)^{-\beta} \Big( g'(x(\hat{s})) - g'(Y(\hat{s})) \Big) \int_0^{\hat{s}} \big( (s-u)^{-\beta} - (\hat{s}-u)^{-\beta} \big) g(x(u)) \rd W(u) \rd W(s), \\
&\cB_{3,2}(t_n) := \int_{0}^{t_n} (t_n-s)^{-\beta} g'(Y(\hat{s})) \int_0^{\hat{s}} \big( (s-u)^{-\beta} - (\hat{s}-u)^{-\beta} \big) \Big( g(x(u)) - g(x(\hat{u})) \Big) \rd W(u) \rd W(s), \\
&\cB_{3,3}(t_n):= \int_{0}^{t_n} (t_n-s)^{-\beta} g'(Y(\hat{s})) \int_0^{\hat{s}} \big( (s-u)^{-\beta} - (\hat{s}-u)^{-\beta} \big) \Big( g(x(\hat{u})) - g(Y(\hat{u})) \Big) \rd W(u) \rd W(s). 
\end{align*}
}It follows from the BDG inequality, H\"older's inequality, the continuity and boundedness of $g''$, Lemma \ref{lem.exMilst} and \eqref{eq.BetaC} that 
\begin{align*}
&\quad\, \big\| \cB_{3,1}(t_n) \big\|_{L^p(\Omega,\hR^d)}^2 \\
&\leq C \int_{0}^{t_n} (t_n-s)^{-2\beta} \big\| x(\hat{s}) - Y(\hat{s}) \big\|_{L^{2p}(\Omega,\hR^d)}^2 \\
&\quad\, \times \Big\| \int_0^{\hat{s}} \big( (s-u)^{-\beta} - (\hat{s}-u)^{-\beta} \big) g(x(u)) \rd W(u) \Big\|_{L^{2p}(\Omega,\hR^d)}^2 \rd s \\
& \leq C \int_{t_1}^{t_n} (t_n-s)^{-2\beta} \big\| x(\hat{s}) - Y(\hat{s}) \big\|_{L^{2p}(\Omega,\hR^d)}^2 \int_0^{\hat{s}} \big| (s-u)^{-\beta} - (\hat{s}-u)^{-\beta} \big|^2\rd u\rd s \\
& \leq C N^{-2(1-2\beta)} \int_{t_1}^{t_n} (t_n-s)^{-2\beta} \hat{s}^{2(1-\alpha+\rho - (\frac{1}{2}-\beta)/r)} \rd s \leq C N^{-2(1-2\beta)}.
\end{align*}

When $n=1$, $\cB_{3,2}(t_n)=\cB_{3,3}(t_n) = 0$. When $n \geq 2$, the BDG inequality, the boundedness of $g'$, Theorem \ref{thm.newRegu}, \eqref{eq.diffSqureDoub1}, \eqref{eq.diffSqureDoub2} and \eqref{eq.BetaC} indicate
\begin{align*}
&\quad\, \big\| \cB_{3,2}(t_n) \big\|_{L^p(\Omega,\hR^d)}^2 \\
&\leq C \int_{t_1}^{t_n} (t_n-s)^{-2\beta} \int_{0}^{\hat{s}} \big| (s-u)^{-\beta} - (\hat{s}-u)^{-\beta} \big|^2 \big\| x(u) - x(\hat{u}) \big\|_{L^p(\Omega,\hR^d)}^2 \rd u \rd s \\
&\leq C N^{-2(1-2\beta)}.
\end{align*}
When $n \geq 2$, the BDG inequality and boundedness of $g'$ show
\begin{align*}
&\quad\, \big\| \cB_{3,3}(t_n) \big\|_{L^p(\Omega,\hR^d)}^2 \\
&\leq C \sum_{i=1}^{n-1} \sum_{j=0}^{i-1} \int_{t_i}^{t_{i+1}} (t_n-s)^{-2\beta} \int_{t_j}^{t_{j+1}} (t_i-u)^{-2\beta} \big\| x(\hat{u}) - Y(\hat{u}) \big\|_{L^p(\Omega,\hR^d)}^2 \rd u \rd s \\
&= C \sum_{j=0}^{n-2} \sum_{i=j+1}^{n-1} \int_{t_i}^{t_{i+1}} (t_n-s)^{-2\beta} \int_{t_j}^{t_{j+1}} (t_i-u)^{-2\beta} \big\| x(\hat{u}) - Y(\hat{u}) \big\|_{L^p(\Omega,\hR^d)}^2 \rd u \rd s \\
&= C \sum_{j=0}^{n-2} \int_{t_{j+1}}^{t_{j+2}} (t_n-s)^{-2\beta} \int_{t_j}^{t_{j+1}} (t_{j+1}-u)^{-2\beta} \big\| x(\hat{u}) - Y(\hat{u}) \big\|_{L^p(\Omega,\hR^d)}^2 \rd u \rd s \\
&\quad\, + C \sum_{j=0}^{n-2} \sum_{i=j+2}^{n-1} \int_{t_j}^{t_{j+1}} \int_{t_i}^{t_{i+1}} (t_n-s)^{-2\beta} (t_i-u)^{-2\beta} \big\| x(\hat{u}) - Y(\hat{u}) \big\|_{L^p(\Omega,\hR^d)}^2 \rd s \rd u \\
&=: C \cB_{3,3}^{\star} + C \cB_{3,3}^{\star\star}.
\end{align*}
By virtue of Lemma \ref{lem.exMilst},
\begin{align*}
\cB_{3,3}^{\star}
&\leq C \sum_{j=1}^{n-2} \int_{t_{j+1}}^{t_{j+2}} (t_n-s)^{-2\beta} \int_{t_j}^{t_{j+1}} (t_{j+1}-u)^{-2\beta} \hat{u}^{2(1-\alpha+\rho - (\frac{1}{2}-\beta)/r)} N^{-2(\frac{1}{2}-\beta)} \rd u \rd s \\
&\leq C N^{-2(1-2\beta)} \sum_{j=1}^{n-2} \int_{t_{j+1}}^{t_{j+2}} (t_n-s)^{-2\beta} s^{2(1-\alpha+\rho - (\frac{1}{2}-\beta)/r)} \rd s \leq C N^{-2(1-2\beta)}.
\end{align*}
On the other hand, by \eqref{eq.BetaInte},
\begin{align*}
\cB_{3,3}^{\star\star}
&\leq C \sum_{j=0}^{n-2} \int_{t_j}^{t_{j+1}} \sum_{i=j+2}^{n-1} \int_{t_i}^{t_{i+1}} (t_n-s)^{-2\beta} (s-u)^{-2\beta} \rd s \big\| x(\hat{u}) - Y(\hat{u}) \big\|_{L^p(\Omega,\hR^d)}^2 \rd u \\
&\leq C \sum_{j=0}^{n-2} \int_{t_j}^{t_{j+1}} (t_n-u)^{1-4\beta} \big\| x(\hat{u}) - Y(\hat{u}) \big\|_{L^p(\Omega,\hR^d)}^2 \rd u. 
\end{align*}
Collecting the estimates of $\cB_{3,3}^{\star}$ and $\cB_{3,3}^{\star\star}$ and using $1-4\beta\ge -2\beta$, we obtain 
\begin{equation*} 
\big\| \cB_3(t_n) \big\|_{L^p(\Omega,\hR^d)}^2
\leq C N^{-2(1-2\beta)} + C \int_{0}^{t_{n}} (t_n-u)^{-2\beta} \big\| x(\hat{u}) - Y(\hat{u}) \big\|_{L^p(\Omega,\hR^d)}^2 \rd u.
\end{equation*}
\textbf{Estimate of $\cB_4(t_n)$:} 
Note that $\cB_4(t_n)=\sum_{l=1}^3\cB_{4,l}(t_n)$ with 
\begin{align*}
&\cB_{4,1}(t_n):= \int_{0}^{t_n} (t_n-s)^{-\beta} \Big( g'(x(\hat{s})) - g'(Y(\hat{s})) \Big) \int_{\hat{s}}^s (s-u)^{-\beta} g(x(u)) \rd W(u) \rd W(s), \\
&\cB_{4,2}(t_n):=\int_{0}^{t_n} (t_n-s)^{-\beta} g'(Y(\hat{s})) \int_{\hat{s}}^s (s-u)^{-\beta} \Big( g(x(u)) - g(x(\hat{u})) \Big) \rd W(u) \rd W(s) , \\
&\cB_{4,3}(t_n):= \int_{0}^{t_n} (t_n-s)^{-\beta} g'(Y(\hat{s})) \int_{\hat{s}}^s (s-u)^{-\beta} \Big( g(x(\hat{u})) - g(Y(\hat{u})) \Big) \rd W(u) \rd W(s).
\end{align*}
Analogously to the estimate of $\| \cB_{3,1}(t_n) \|_{L^p(\Omega,\hR^d)}^2$, one can get 
\begin{align*}
\big\| \cB_{4,1}(t_n) \big\|_{L^p(\Omega,\hR^d)}^2 \leq C N^{-2(1-2\beta)}. 
\end{align*}
In virtue of the BDG inequality and \eqref{eq.BetaC}, one has 
\begin{align*}
&\quad\, \big\| \cB_{4,2}(t_n) \big\|_{L^p(\Omega,\hR^d)}^2 \\
&\leq C \int_{0}^{t_1} (t_n-s)^{-2\beta} \int_{\hat{s}}^s (s-u)^{-2\beta} \big\| x(u) - x(\hat{u}) \big\|_{L^p(\Omega,\hR^d)}^2 \rd u \rd s \\
&\quad + C \int_{t_1}^{t_n} (t_n-s)^{-2\beta} \int_{\hat{s}}^s (s-u)^{-2\beta} \big\| x(u) - x(\hat{u}) \big\|_{L^p(\Omega,\hR^d)}^2 \rd u \rd s \\
& \leq C N^{-2(1-2\beta)} + C \int_{t_1}^{t_n} (t_n-s)^{-2\beta} \int_{\hat{s}}^s (s-u)^{-2\beta} \hat{u}^{2\beta - 1} (u-\hat{u})^{1-2\beta} \rd u \rd s \\
& \leq C N^{-2(1-2\beta)} + C N^{-2(1-2\beta)} \int_{t_1}^{t_n} (t_n-s)^{-2\beta} \hat{s}^{2\beta - 1} \rd s \leq C N^{-2(1-2\beta)}.
\end{align*}
Using the BDG inequality and the boundedness of $g'$ shows
\begin{align*}
\big\| \cB_{4,3}(t_n) \big\|_{L^p(\Omega,\hR^d)}^2 &\leq C \int_{0}^{t_n} (t_n-s)^{-2\beta} \int_{\hat{s}}^s (s-r)^{-2\beta} \big\| x(\hat{u}) - Y(\hat{u}) \big\|_{L^p(\Omega,\hR^d)}^2 \rd u \rd s \\
&\leq C \int_{0}^{t_n} (t_n-s)^{-2\beta} \big\| x(\hat{s}) - Y(\hat{s}) \big\|_{L^p(\Omega,\hR^d)}^2 \rd s.
\end{align*}
Gathering the above estimates leads to
\begin{equation*} 
\big\| \cB_4(t_n) \big\|_{L^p(\Omega,\hR^d)}^2
\leq C N^{-2(1-2\beta)} + C \int_{0}^{t_{n}} (t_n-u)^{-2\beta} \big\| x(\hat{u}) - Y(\hat{u}) \big\|_{L^p(\Omega,\hR^d)}^2 \rd u.
\end{equation*}

Taking Lemma \ref{lem.Ak} into account, it holds that
\begin{align*}
\big\| x(t_n) - Y(t_n) \big\|_{L^p(\Omega,\hR^d)}^2
&\leq C t_n^{2(\frac{1}{2}-\alpha+\beta+\sigma - \sigma/r)} N^{ -2\min\{\frac{3}{2}-\alpha-\beta,\, 1-2\beta\} } \\
&\quad + C \int_{0}^{t_n} (t_n-s)^{-\max\{ \alpha, 2\beta \}} \big\| x(\hat{s}) - Y(\hat{s}) \big\|_{L^p(\Omega,\hR^d)}^2 \rd s.
\end{align*}
Finally, applying the Gr\"onwall inequality (i.e., Lemma \ref{lem.Gronwall}) completes the proof.
\end{proof}

\section{Numerical experiments}
\label{sec.NumerExper}

In this section, we report numerical results for the following illustrative example
\begin{align} \label{eq.numericalExample}
x(t) &= 1 - (1-\alpha) \int_0^t (t-s)^{-\alpha} \sin\big(\frac{1}{2}x(s)\big) \rd s \notag\\
&\quad + \int_0^t (t-s)^{-\beta} \cos\big(\frac{1}{2}x(s)\big) \rd W(s), \quad t \in [0,1],
\end{align}
in order to verify the theoretical results in section \ref{sec.NumericalMethods}. Notice that the drift and diffusion coefficients in \eqref{eq.numericalExample} satisfy the conditions in Theorems \ref{thm.EM}, \ref{thm.fastEM} and \ref{thm.Milstein}. Correspond to the obtained theoretical results, we will consider two kinds of errors:\
\begin{align*}
 \text{Err}_{\text{end}} := \big\| x(t_N) - Z_N \big\|_{L^2(\Omega,\hR)} \quad \mbox{and} \quad \text{Err}_{\max} := \sup_{1 \leq n \leq N} \big\| x(t_n) - Z_n \big\|_{L^2(\Omega,\hR)},
\end{align*}
where $Z$ denotes the numerical solution computed by the EM method \eqref{eq.EMscheme}, the fast EM method \eqref{eq.fastEMscheme} or the Milstein method \eqref{eq.Milstein}. We take the numerical solution calculated by a very fine mesh with $N = 2^{13}$ as the unknown `exact' solution. The Monte Carlo method with $5, 000$ sample paths is used to simulate expectations. And we adopt the least squares fit to obtain the order of numerical errors (see e.g., \cite{Higham2001SIREV}).

In Tables \ref{tab.EM1} and \ref{tab.EM2} (resp.\ Tables \ref{tab.fastEM1} and \ref{tab.fastEM2}), we list the error and convergence order of the EM method (resp.\ the fast EM method), which are consistent with Theorem \ref{thm.EM} (resp.\ Theorem \ref{thm.fastEM}). Table \ref{tab.Test1CPUTime} displays the CPU time of the EM method and the fast EM method, which shows that the fast EM method has a better performance than the EM method. Notice that the implementation of the Milstein method \eqref{eq.Milstein} requires the simulation of the multiple stochastic integral with singular kernels
\begin{align} \label{eq.douSinStoInte}
\sum_{i=0}^{n-1} \int_{t_{i}}^{t_{i+1}} (t_n-s)^{-\beta} g'(Y_i) \bigg( \sum_{j=0}^{i-1} \int_{t_j}^{t_{j+1}} \big( (s-u)^{-\beta} - (t_i-u)^{-\beta} \big) g(Y_j) \rd W(u) \bigg) \rd W(s).
\end{align}
Since there is currently no efficient way for simulating \eqref{eq.douSinStoInte} with $\beta \in (0,\frac{1}{2})$, we take $\beta = 0$ to test the convergence order of the Milstein method. The corresponding numerical results are compatible with Theorem \ref{thm.Milstein}, as presented in Tables \ref{tab.Milstein1} and \ref{tab.Milstein2}.

\begin{table}
\centering
\tabcolsep = 0.38cm
\caption{\label{tab.EM1} The error and convergence order of the EM method \eqref{eq.EMscheme} when $\alpha = 0.9$ and $\beta = 0.1$.}
\begin{tabularx}{\textwidth}{ccccc}
\toprule
\ \
\multirow{2}{*}{$N$}
& \multicolumn{2}{l}{\makecell[c]{$r = 1$}} & \multicolumn{2}{l}{\makecell[c]{$r = 2$}} \\
 \cmidrule(lr){2-3} \cmidrule(lr){4-5} & $\text{Err}_{\text{end}}$ & $\text{Err}_{\max}$ & $\text{Err}_{\text{end}}$ & $\text{Err}_{\max}$ \\
\midrule
\ \ $2^{7}$ & $2.8261e$-$02$ & $6.6831e$-$02$ & $3.7934e$-$02$ & $3.8486e$-$02$ \\
\ \ $2^{8}$ & $2.0945e$-$02$ & $5.7599e$-$02$ & $2.8075e$-$02$ & $2.8349e$-$02$ \\
\ \ $2^{9}$ & $1.5480e$-$02$ & $5.0155e$-$02$ & $2.0450e$-$02$ & $2.0661e$-$02$ \\
\ \ $2^{10}$ & $1.1396e$-$02$ & $4.3819e$-$02$ & $1.5168e$-$02$ & $1.5309e$-$02$ \\
\midrule
\ \ \mbox{Numerical } & \multirow{2}{*}{$0.4367$} & \multirow{2}{*}{$0.2027$} & \multirow{2}{*}{$0.4425$} & \multirow{2}{*}{$0.4446$}
\\\mbox{ Order} \\
\midrule
\ \ \mbox{Theoretical } & \multirow{2}{*}{$0.4$} & \multirow{2}{*}{$0.2$} & \multirow{2}{*}{$0.4$} & \multirow{2}{*}{$0.4$}
\\\mbox{ Order} \\
\bottomrule
\end{tabularx}
\end{table}

\begin{table}
\centering
\tabcolsep = 0.38cm
\caption{\label{tab.EM2} The error and convergence order of the EM method \eqref{eq.EMscheme} when $\alpha = 0.8$ and $\beta = 0.1$.}
\begin{tabularx}{\textwidth}{ccccc}
\toprule
\ \
\multirow{2}{*}{$N$}
& \multicolumn{2}{l}{\makecell[c]{$r = 1$}} & \multicolumn{2}{l}{\makecell[c]{$r = 2$}} \\
 \cmidrule(lr){2-3} \cmidrule(lr){4-5} & $\text{Err}_{\text{end}}$ & $\text{Err}_{\max}$ & $\text{Err}_{\text{end}}$ & $\text{Err}_{\max}$ \\
\midrule
\ \ $2^{7}$ & $2.3741e$-$02$ & $3.0558e$-$02$ & $3.0782e$-$02$ & $3.0981e$-$02$ \\
\ \ $2^{8}$ & $1.7359e$-$02$ & $2.2628e$-$02$ & $2.2374e$-$02$ & $2.2374e$-$02$ \\
\ \ $2^{9}$ & $1.2736e$-$02$ & $1.6976e$-$02$ & $1.6043e$-$02$ & $1.6184e$-$02$ \\
\ \ $2^{10}$ & $9.0637e$-$03$ & $1.2835e$-$02$ & $1.1466e$-$02$ & $1.1535e$-$02$ \\
\midrule
\ \ \mbox{Numerical } & \multirow{2}{*}{$0.4614$} & \multirow{2}{*}{$0.4169$} & \multirow{2}{*}{$0.4754$} & \multirow{2}{*}{$0.4743$}
\\\mbox{ Order} \\
\midrule
\ \ \mbox{Theoretical } & \multirow{2}{*}{$0.4$} & \multirow{2}{*}{$0.4$} & \multirow{2}{*}{$0.4$} & \multirow{2}{*}{$0.4$}
\\\mbox{ Order} \\
\bottomrule
\end{tabularx}
\end{table}

\begin{table}
\centering
\tabcolsep = 0.38cm
\caption{\label{tab.fastEM1} The error and convergence order of the fast EM method \eqref{eq.fastEMscheme} with $\epsilon = 10^{-6}$ when $\alpha = 0.9$ and $\beta = 0.1$.}
\begin{tabularx}{\textwidth}{ccccc}
\toprule
\ \
\multirow{2}{*}{$N$}
& \multicolumn{2}{l}{\makecell[c]{$r = 1$}} & \multicolumn{2}{l}{\makecell[c]{$r = 2$}} \\
 \cmidrule(lr){2-3} \cmidrule(lr){4-5} & $\text{Err}_{\text{end}}$ & $\text{Err}_{\max}$ & $\text{Err}_{\text{end}}$ & $\text{Err}_{\max}$ \\
\midrule
\ \ $2^{7}$ & $2.8261e$-$02$ & $6.6831e$-$02$ & $3.7934e$-$02$ & $3.8486e$-$02$ \\
\ \ $2^{8}$ & $2.0945e$-$02$ & $5.7599e$-$02$ & $2.8075e$-$02$ & $2.8349e$-$02$ \\
\ \ $2^{9}$ & $1.5480e$-$02$ & $5.0155e$-$02$ & $2.0450e$-$02$ & $2.0661e$-$02$ \\
\ \ $2^{10}$ & $1.1396e$-$02$ & $4.3819e$-$02$ & $1.5168e$-$02$ & $1.5309e$-$02$ \\
\midrule
\ \ \mbox{Numerical } & \multirow{2}{*}{$0.4367$} & \multirow{2}{*}{$0.2027$} & \multirow{2}{*}{$0.4425$} & \multirow{2}{*}{$0.4446$}
\\\mbox{ Order} \\
\midrule
\ \ \mbox{Theoretical } & \multirow{2}{*}{$0.4$} & \multirow{2}{*}{$0.2$} & \multirow{2}{*}{$0.4$} & \multirow{2}{*}{$0.4$}
\\\mbox{ Order} \\
\bottomrule
\end{tabularx}
\end{table}

\begin{table}
\centering
\tabcolsep = 0.38cm
\caption{\label{tab.fastEM2} The error and convergence order of the fast EM method \eqref{eq.fastEMscheme} with $\epsilon = 10^{-6}$ when $\alpha = 0.8$ and $\beta = 0.1$.}
\begin{tabularx}{\textwidth}{ccccc}
\toprule
\ \
\multirow{2}{*}{$N$}
& \multicolumn{2}{l}{\makecell[c]{$r = 1$}} & \multicolumn{2}{l}{\makecell[c]{$r = 2$}} \\
 \cmidrule(lr){2-3} \cmidrule(lr){4-5} & $\text{Err}_{\text{end}}$ & $\text{Err}_{\max}$ & $\text{Err}_{\text{end}}$ & $\text{Err}_{\max}$ \\
\midrule
\ \ $2^{7}$ & $2.3741e$-$02$ & $3.0558e$-$02$ & $3.0782e$-$02$ & $3.0981e$-$02$ \\
\ \ $2^{8}$ & $1.7359e$-$02$ & $2.2628e$-$02$ & $2.2374e$-$02$ & $2.2374e$-$02$ \\
\ \ $2^{9}$ & $1.2736e$-$02$ & $1.6976e$-$02$ & $1.6043e$-$02$ & $1.6184e$-$02$ \\
\ \ $2^{10}$ & $9.0637e$-$03$ & $1.2835e$-$02$ & $1.1466e$-$02$ & $1.1535e$-$02$ \\
\midrule
\ \ \mbox{Numerical } & \multirow{2}{*}{$0.4614$} & \multirow{2}{*}{$0.4169$} & \multirow{2}{*}{$0.4754$} & \multirow{2}{*}{$0.4743$}
\\\mbox{ Order} \\
\midrule
\ \ \mbox{Theoretical } & \multirow{2}{*}{$0.4$} & \multirow{2}{*}{$0.4$} & \multirow{2}{*}{$0.4$} & \multirow{2}{*}{$0.4$}
\\\mbox{ Order} \\
\bottomrule
\end{tabularx}
\end{table}

\begin{table}
\centering
\tabcolsep = 0.13cm
\renewcommand{\arraystretch}{1.15}
\caption{\label{tab.Test1CPUTime} The CPU time(s) of the EM method \eqref{eq.EMscheme} and the fast EM method \eqref{eq.fastEMscheme} when $r = 2$, $\alpha = 0.9$, and $\beta = 0.1$.}
\begin{tabularx}{\textwidth}{cccccc}
\toprule
\ \ \multirow{2}{*}{$N$} & \multirow{2}{*}{$\quad$EM method$\quad$} & \multicolumn{4}{l}{\makecell[c]{fast EM method}} \\
 \cmidrule(lr){3-6} & & $\ \ \epsilon = 10^{-6}\ \ $ & $\ \ \epsilon = 10^{-9}\ \ $ & $\ \ \ \epsilon = 10^{-12}\ \ $ & $\ \ \ \epsilon = 10^{-15}\ \ $ \\
\midrule
\ \ $2^{7}$ & $0.974310$ & $0.249094$ & $0.477652$ & $0.567012$ & $0.669200$ \\
\ \ $2^{8}$ & $3.815990$ & $0.771575$ & $0.900727$ & $1.200556$ & $1.397880$ \\
\ \ $2^{9}$ & $14.714308$ & $1.549202$ & $2.004160$ & $2.591853$ & $2.951605$ \\
\ \ $2^{10}$ & $54.902414$ & $3.219936$ & $4.256469$ & $5.675455$ & $6.295207$ \\
\ \ $2^{11}$ & $238.854192$ & $6.750346$ & $8.918369$ & $11.092695$ & $13.060386$ \\
\bottomrule
\end{tabularx}
\end{table}

\begin{table}
\centering
\tabcolsep = 0.38cm
\caption{\label{tab.Milstein1} The error and convergence order of the Milstein method \eqref{eq.Milstein} when $\alpha = 0.9$ and $\beta = 0$.}
\begin{tabularx}{\textwidth}{ccccc}
\toprule
\ \
\multirow{2}{*}{$N$}
& \multicolumn{2}{l}{\makecell[c]{$r = 1$}} & \multicolumn{2}{l}{\makecell[c]{$r = 3$}} \\
 \cmidrule(lr){2-3} \cmidrule(lr){4-5} & $\text{Err}_{\text{end}}$ & $\text{Err}_{\max}$ & $\text{Err}_{\text{end}}$ & $\text{Err}_{\max}$ \\
\midrule
\ \ $2^{7}$ & $1.7249e$-$02$ & $6.5181e$-$02$ & $3.2115e$-$02$ & $3.2269e$-$02$ \\
\ \ $2^{8}$ & $1.1579e$-$02$ & $5.6693e$-$02$ & $2.1550e$-$02$ & $2.1550e$-$02$ \\
\ \ $2^{9}$ & $7.6257e$-$03$ & $4.9459e$-$02$ & $1.4274e$-$02$ & $1.4274e$-$02$ \\
\ \ $2^{10}$ & $4.9995e$-$03$ & $4.3103e$-$02$ & $9.4015e$-$03$ & $9.4261e$-$03$ \\
\midrule
\ \ \mbox{Numerical } & \multirow{2}{*}{$0.5963$} & \multirow{2}{*}{$0.1987$} & \multirow{2}{*}{$0.5911$} & \multirow{2}{*}{$0.5921$}
\\\mbox{ Order} \\
\midrule
\ \ \mbox{Theoretical } & \multirow{2}{*}{$0.6$} & \multirow{2}{*}{$0.2$} & \multirow{2}{*}{$0.6$} & \multirow{2}{*}{$0.6$}
\\\mbox{ Order} \\
\bottomrule
\end{tabularx}
\end{table}

\begin{table}
\centering
\tabcolsep = 0.38cm
\caption{\label{tab.Milstein2} The error and convergence order of the Milstein method \eqref{eq.Milstein} when $\alpha = 0.5$ and $\beta = 0$.}
\begin{tabularx}{\textwidth}{ccccc}
\toprule
\ \
\multirow{2}{*}{$N$}
& \multicolumn{2}{l}{\makecell[c]{$r = 1$}} & \multicolumn{2}{l}{\makecell[c]{$r = 3$}} \\
 \cmidrule(lr){2-3} \cmidrule(lr){4-5} & $\text{Err}_{\text{end}}$ & $\text{Err}_{\max}$ & $\text{Err}_{\text{end}}$ & $\text{Err}_{\max}$ \\
\midrule
\ \ $2^{7}$ & $2.7171e$-$03$ & $2.8104e$-$03$ & $7.1213e$-$03$ & $7.1213e$-$03$ \\
\ \ $2^{8}$ & $1.4205e$-$03$ & $1.4578e$-$03$ & $3.7561e$-$03$ & $3.7561e$-$03$ \\
\ \ $2^{9}$ & $7.1720e$-$04$ & $7.4171e$-$04$ & $1.9217e$-$03$ & $1.9246e$-$03$ \\
\ \ $2^{10}$ & $3.5872e$-$04$ & $3.7156e$-$04$ & $9.7635e$-$04$ & $9.7635e$-$04$ \\
\midrule
\ \ \mbox{Numerical } & \multirow{2}{*}{$0.9749$} & \multirow{2}{*}{$0.9732$} & \multirow{2}{*}{$0.9783$} & \multirow{2}{*}{$0.9783$}
\\\mbox{ Order} \\
\midrule
\ \ \mbox{Theoretical } & \multirow{2}{*}{$1.0$} & \multirow{2}{*}{$1.0$} & \multirow{2}{*}{$1.0$} & \multirow{2}{*}{$1.0$}
\\\mbox{ Order} \\
\bottomrule
\end{tabularx}
\end{table}

\appendix

\section{Proof of the Gr\"onwall-type inequality}
\label{appendix.Gronwall}

To facilitate the proof of Lemma \ref{lem.Gronwall}, we first prove the special case $\gamma = 1$.

\begin{lemma}\label{lem.GW1}
Let $\mu \in (0,1]$, $C_1 > 0$, $C_2 > 0$, $t_n \in \cM_r$ with $n \in \{1,2,\ldots,N\}$ and $r \geq 1$. If the non-negative sequence $\{z_n\}_{n=1}^N$ satisfies
\begin{align*}
z_n \leq C_1 t_n^{\mu-1} + C_2 \sum_{j=1}^{n-1} \int_{t_j}^{t_{j+1}} z_j \rd s \quad \forall\, n \in \{1,2,\ldots,N\},
\end{align*}
then $z_n\leq (1+ C_2T\mu^{-1}e^{C_2T}) C_1 t_{n}^{\mu-1}$ for all $n \in \{1,2,\ldots,N\}.$
\end{lemma}

\begin{proof}
Iteratively using the condition of this lemma yields
\begin{align*}
z_n&\leq C_1 t_n^{\mu-1}+C_2 \sum_{j_1 = 1}^{n-1} \int_{t_{j_1}}^{t_{j_1+1}} z_{j_1} \rd s_1\\
&\leq C_1t_n^{\mu-1} + C_1 C_2 \sum_{j_1 = 1}^{n-1} \int_{t_{j_1}}^{t_{j_1+1}} t_{j_1}^{\mu-1} \rd s_1+C_2^2 \sum_{j_1 = 1}^{n-1} \int_{t_{j_1}}^{t_{j_1+1}} \sum_{j_2 = 1}^{j_1-1} \int_{t_{j_2}}^{t_{j_2+1}} z_{j_2} \rd s_2 \rd s_1 \\
&\cdots\\
&\leq C_1t_n^{\mu-1}+C_1\sum_{k=1}^{n-1} C_2^k \sum_{j_1 = 1}^{n-1} \int_{t_{j_1}}^{t_{j_1+1}}\sum_{j_2 = 1}^{j_1-1} \int_{t_{j_2}}^{t_{j_2+1}} \cdots\sum_{j_k = 1}^{j_{k-1}-1}\int_{t_{j_k}}^{t_{j_k+1}} t_{j_k}^{\mu-1} \rd s_k \cdots \rd s_2 \rd s_1 .
\end{align*}
Here, it follows from \eqref{eq.s} that for all $k \in \{1,2,\ldots,n-1\}$,
\begin{align*}
&\quad\, \sum_{j_1 = 1}^{n-1} \int_{t_{j_1}}^{t_{j_1+1}}\sum_{j_2 = 1}^{j_1-1} \int_{t_{j_2}}^{t_{j_2+1}} \cdots\sum_{j_k = 1}^{j_{k-1}-1}\int_{t_{j_k}}^{t_{j_k+1}} t_{j_k}^{\mu-1} \rd s_k \cdots \rd s_2 \rd s_1 \\
&\leq \sum_{j_1 = 1}^{n-1} \int_{t_{j_1}}^{t_{j_1+1}}\sum_{j_2 = 1}^{j_1-1} \int_{t_{j_2}}^{t_{j_2+1}} \cdots\sum_{j_k = 1}^{j_{k-1}-1}\int_{t_{j_k}}^{t_{j_k+1}} s_k^{\mu-1} \rd s_k \cdots \rd s_2 \rd s_1 \\
&\leq \frac{1}{\mu}\sum_{j_1 = 1}^{n-1} \int_{t_{j_1}}^{t_{j_1+1}}\sum_{j_2 = 1}^{j_1-1} \int_{t_{j_2}}^{t_{j_2+1}} \cdots\sum_{j_{k-1} = 1}^{j_{k-2}-1}\int_{t_{j_{k-1}}}^{t_{j_{k-1}+1}} t_{j_{k-1}}^{\mu} \rd s_{k-1} \cdots \rd s_2 \rd s_1 \\
&\leq \frac{1}{\mu}\sum_{j_1 = 1}^{n-1} \int_{t_{j_1}}^{t_{j_1+1}}\sum_{j_2 = 1}^{j_1-1} \int_{t_{j_2}}^{t_{j_2+1}} \cdots\sum_{j_{k-1} = 1}^{j_{k-2}-1}\int_{t_{j_{k-1}}}^{t_{j_{k-1}+1}} s_{k-1}^{\mu} \rd s_{k-1} \cdots \rd s_2 \rd s_1 \\
&\leq \frac{1}{\mu(\mu+1)}\sum_{j_1 = 1}^{n-1} \int_{t_{j_1}}^{t_{j_1+1}}\sum_{j_2 = 1}^{j_1-1} \int_{t_{j_2}}^{t_{j_2+1}} \cdots\sum_{j_{k-2} = 1}^{j_{k-3}-1}\int_{t_{j_{k-2}}}^{t_{j_{k-2}+1}} t_{j_{k-2}}^{\mu+1} \rd s_{k-2} \cdots \rd s_2 \rd s_1 \\
&\leq \cdots\\
&\leq \frac{1}{\mu(\mu+1)\cdots(\mu+k-2)}\sum_{j_1 = 1}^{n-1} \int_{t_{j_1}}^{t_{j_1+1}} t_{j_{1}}^{\mu+k-2} \rd s_1\\
&\leq \frac{1}{\mu(\mu+1)\cdots(\mu+k-1)} t_{n}^{\mu+k-1}.
 \end{align*}
Thus, one can deduce
\begin{align*}
z_n
&\leq C_1t_n^{\mu-1}+C_1\sum_{k=1}^{n-1} \frac{C_2^k}{\mu(\mu+1)\cdots(\mu+k-1)} t_{n}^{\mu+k-1} \\
&\leq C_1t_n^{\mu-1}+C_1 t_{n}^{\mu-1}\sum_{k=1}^{n-1} \frac{C_2^kT^k}{\mu(\mu+1)\cdots(\mu+k-1)}\\
&\leq C_1t_n^{\mu-1}+C_1 C_2T\mu^{-1}t_{n}^{\mu-1}\sum_{k=1}^{n-1} \frac{(C_2T)^{k-1}}{(k-1)!}\\
&\leq C_1t_n^{\mu-1}+C_1 C_2T\mu^{-1}t_{n}^{\mu-1}e^{C_2T},
\end{align*}
which completes the proof.
\end{proof}

\begin{proof}[Proof of Lemma \ref{lem.Gronwall}]
Let $n \in \{2,3,\ldots,N\}$ and $m \in \{1,2,\ldots,n-1\}$ be arbitrary. It follows from \eqref{eq.BetaC} and $t_n-t_1\ge\frac{1}{2}t_n$ that
\begin{align} 
\sum_{j=1}^{n-m} \int_{t_j}^{t_{j+1}} (t_n-s)^{m\gamma-1} t_j^{\mu-1} \rd s
&\leq 2^{r(1-\mu)} B(m\gamma,\mu) (t_n-t_1)^{m\gamma+\mu-1} \notag\\
&\leq \widetilde{B}(m\gamma,\mu) t_n^{m\gamma+\mu-1}, \label{eq.tidleB}
\end{align}
where $\widetilde{B}(m\gamma,\mu) := 2^{r(1-\mu)} B(m\gamma,\mu)\max\{2^{1-m\gamma-\mu},1\}$. Notice that
\begin{align}
&\quad\, \sum_{j=1}^{n-m} \int_{t_j}^{t_{j+1}} (t_n-s)^{m\gamma-1} \left( \sum_{k=1}^{j-1} \int_{t_k}^{t_{k+1}} (t_j-u)^{\gamma-1} z_k \rd u \right) \rd s \notag \\
&= \sum_{k=1}^{n-m-1} \int_{t_k}^{t_{k+1}} \bigg( \sum_{j=k+1}^{n-m} \int_{t_j}^{t_{j+1}} (t_n-s)^{m\gamma-1} (t_j-u)^{\gamma-1} \rd s \bigg) z_k \rd u \notag \\
&= \sum_{k=1}^{n-m-1} \int_{t_k}^{t_{k+1}} \bigg( \int_{t_{k+1}}^{t_{k+2}} (t_n-s)^{m\gamma-1} (t_{k+1}-u)^{\gamma-1} \rd s \bigg) z_k \rd u \notag \\
&\quad + \sum_{k=1}^{n-m-1} \int_{t_k}^{t_{k+1}} \bigg( \sum_{j=k+2}^{n-m} \int_{t_j}^{t_{j+1}} (t_n-s)^{m\gamma-1} (t_j-u)^{\gamma-1} \rd s \bigg) z_k \rd u \notag \\
&=: \cK_1 + \cK_2. \label{eq.cK0}
\end{align}
For all $k \in \{1,2,\ldots,n-m-1\}$, $j \in \{k+2, k+3, \ldots, n-m\}$, $u\in(t_k,t_{k+1})$ and $s\in(t_j,t_{j+1})$, it follows from \eqref{eq.LaterStepFormerStep} that $s-u \leq (1+3^{r-1}) (t_j-u)$. Thus, by \eqref{eq.BetaInte},
\begin{align*}
&\quad\, \sum_{j=k+2}^{n-m} \int_{t_j}^{t_{j+1}} (t_n-s)^{m\gamma-1} (t_j-u)^{\gamma-1} \rd s \\
&\leq (1+3^{r-1})^{1-\gamma} \sum_{j=k+2}^{n-m} \int_{t_j}^{t_{j+1}} (t_n-s)^{m\gamma-1} (s-u)^{\gamma-1} \rd s\\
&\leq (1+3^{r-1})^{1-\gamma} \int_{u}^{t_n} (t_n-s)^{m\gamma-1} (s-u)^{\gamma-1} \rd s\\
&= (1+3^{r-1})^{1-\gamma} B(m\gamma,\gamma) (t_n-u)^{(m+1)\gamma-1},
\end{align*}
which reads
\begin{align}\label{eq.cK2}
\cK_2
\leq (1+3^{r-1})^{1-\gamma} B(m\gamma,\gamma)\sum_{k=1}^{n-m-1} \int_{t_k}^{t_{k+1}} (t_n-s)^{(m+1)\gamma-1} \rd s\, z_k.
\end{align} 
We claim that for any $1 \leq k \leq n-2$, 
\begin{align} \label{eq.calimGronwall}
\int_{t_{k+1}}^{t_{k+2}} (t_n-s)^{m\gamma-1} \rd s 
\leq C(r,m,\gamma) \int_{t_{k}}^{t_{k+1}} (t_{n-1}-s)^{m\gamma-1} \rd s 
\end{align}
with $C(r,m,\gamma) := r 3^{r-1} \max\{ (1+r 3^{r-1})^{m\gamma-1}, 2^{1-m\gamma} \}$. To prove \eqref{eq.calimGronwall}, by the change of variables, we formulate 
\begin{align*}
\int_{t_{k+1}}^{t_{k+2}} (t_n-s)^{m\gamma-1} \rd s
&= \int_{0}^{1} (t_n - t_{k+1} - h_{k+2}\tau)^{m\gamma-1} h_{k+2} \rd \tau, \\
\int_{t_{k}}^{t_{k+1}} (t_{n-1}-s)^{m\gamma-1} \rd s
&= \int_{0}^{1} (t_{n-1} - t_{k} - h_{k+1}\tau)^{m\gamma-1} h_{k+1} \rd \tau. 
\end{align*} 
It follows from \eqref{eq.LaterStepFormerStep} that $h_{k+2} \leq r 3^{r-1} h_{k+1}$, which implies that the proof of \eqref{eq.calimGronwall} boils down to prove that for any $1 \leq k \leq n-2$, 
\begin{align} \label{eq.aimMGamma}
\int_{0}^{1} (t_n - t_{k+1} - h_{k+2}\tau)^{m\gamma-1} \rd \tau 
\leq \widetilde{C}(r,m,\gamma) \int_{0}^{1} (t_{n-1} - t_{k} - h_{k+1}\tau)^{m\gamma-1} \rd \tau
\end{align}
with $\widetilde{C}(r,m,\gamma) := \max\{ (1+r 3^{r-1})^{m\gamma-1}, 2^{1-m\gamma} \}$. Recall that $h_{n-1} \leq h_{n} \leq r 3^{r-1} h_{n-1}$ due to \eqref{eq.LaterStepFormerStep}. When $k = n-2$, one can get 
\begin{align*}
\int_{0}^{1} (h_n - h_{n}\tau)^{m\gamma-1} \rd \tau 
\leq \max\{ (r 3^{r-1})^{m\gamma-1}, 1 \} \int_{0}^{1} (h_{n-1} - h_{n-1}\tau)^{m\gamma-1} \rd \tau.
\end{align*}
This shows that \eqref{eq.aimMGamma} holds for the case $k = n-2$. When $1 \leq k \leq n-3$, one has
\begin{align} \label{eq.mgamma>1}
t_n - t_{k+1}
&= t_{n-1} - t_{k+1} + h_n 
\leq t_{n-1} - t_{k+1} + r 3^{r-1} h_{n-1} \notag\\
&\leq (1+r 3^{r-1}) (t_{n-1} - t_{k+1})
\end{align}
and 
\begin{align} \label{eq.mgamma<1}
t_{n-1} - t_{k} 
&= t_{n} - t_{k+2} - t_{n} + t_{k+2} + t_{n-1} - t_{k} \notag\\
&= t_{n} - t_{k+2} + t_{k+2} - t_{k} - h_n 
\leq 2(t_{n} - t_{k+2}). 
\end{align} 
In fact, \eqref{eq.mgamma>1} and \eqref{eq.mgamma<1} indicate that for any $\tau \in [0,1]$, 
\begin{align*}
t_n - t_{k+1} - h_{k+2} \tau
&\leq t_n - t_{k+1} 
\leq (1+r 3^{r-1}) (t_{n-1} - t_{k+1}) \\
&\leq (1+r 3^{r-1}) (t_{n-1} - t_{k} - h_{k+1}\tau)
\end{align*}
and 
\begin{align*}
t_n - t_{k+1} - h_{k+2}\tau 
&\geq t_n - t_{k+1} - h_{k+2} 
= t_n - t_{k+2} \\
&\geq \frac{1}{2} (t_{n-1} - t_{k}) \geq \frac{1}{2} (t_{n-1} - t_{k} - h_{k+1}\tau), 
\end{align*}
respectively. Then \eqref{eq.aimMGamma} holds for the case $1 \leq k \leq n-3$. Hence the claim \eqref{eq.calimGronwall} holds. By \eqref{eq.LaterStepFormerStep}, one can also obtain
\begin{equation*} 
\begin{aligned}
&\quad\, \int_{t_{k}}^{t_{k+1}} (t_{n-1}-s)^{m\gamma-1} \rd s \\
&\leq \left\{
\begin{split}
& \max \Big\{ (1+3^{r-1})^{1-m\gamma},\, 1 \Big\} \int_{t_{k}}^{t_{k+1}} (t_{n}-s)^{m\gamma-1} \rd s, & \text{if } 1\le k\le n-3, \\
& \max \Big\{ \frac{(1+ \max\{r 2^{r-1},\, 3^{r-1}\})^{1-m\gamma}}{m\gamma}, 1 \Big\} \int_{t_{k}}^{t_{k+1}} (t_{n}-s)^{m\gamma-1} \rd s, \!\!\!\!\!\!\!\!\!& \text{if } k= n-2, 
\end{split}
\right.
\end{aligned}
\end{equation*}
which together with \eqref{eq.calimGronwall} reads 
\begin{align}
\cK_1
&\leq \frac1\gamma \sum_{k=1}^{n-m-1}(t_{k+1}-t_k)^{\gamma} \int_{t_{k+1}}^{t_{k+2}} (t_n-s)^{m\gamma-1} \rd s\, z_k \notag\\
&\leq C(r,m,\gamma) \frac1\gamma \sum_{k=1}^{n-m-1}(t_{k+1}-t_k)^{\gamma} \int_{t_{k}}^{t_{k+1}} (t_{n-1}-s)^{m\gamma-1} \rd s\, z_k \notag\\
&\leq C(r,m,\gamma) \frac1\gamma \max \Big\{ 3^r,\, \frac{r3^r}{m\gamma} \Big\} \sum_{k=1}^{n-m-1}(t_{k+1}-t_k)^{\gamma} \int_{t_{k}}^{t_{k+1}} (t_{n}-s)^{m\gamma-1} \rd s\, z_k \notag\\
&\leq C(r,m,\gamma) \frac{3^r}{\gamma} \max \Big\{ 1,\, \frac{r}{m\gamma} \Big\} \sum_{k=1}^{n-m-1} \int_{t_{k}}^{t_{k+1}} (t_{n}-s)^{(m+1)\gamma-1} \rd s\, z_k. \label{eq.cK1}
\end{align}
Collecting \eqref{eq.cK0}, \eqref{eq.cK2} and \eqref{eq.cK1} leads to
\begin{align}
&\quad\, \sum_{j=1}^{n-m} \int_{t_j}^{t_{j+1}} (t_n-s)^{m\gamma-1} \left( \sum_{k=1}^{j-1} \int_{t_k}^{t_{k+1}} (t_j-u)^{\gamma-1} z_k \rd u \right) \rd s \notag\\
&\le\widehat{B}(m\gamma,\gamma)\sum_{k=1}^{n-m-1} \int_{t_{k}}^{t_{k+1}} (t_{n}-s)^{(m+1)\gamma-1} \rd s\, z_k, \label{eq.cK}
\end{align}
where $\widehat{B}(m\gamma,\gamma) := C(r,m,\gamma) 3^r\gamma^{-1} \max \big\{1,\frac{r}{m\gamma} \big\} + (1+3^{r-1})^{1-\gamma} B(m\gamma,\gamma)$.
By iteration, \eqref{eq.tidleB} and \eqref{eq.cK},
\begin{align*}
z_n
&\leq C_1 t_n^{\mu-1} + C_2 \sum_{j=1}^{n-1} \int_{t_j}^{t_{j+1}} (t_n-s)^{\gamma-1} z_j \rd s \\
&\leq C_1 t_n^{\mu-1} + C_2 \sum_{j=1}^{n-1} \int_{t_j}^{t_{j+1}} (t_n-s)^{\gamma-1} \rd s \left( C_1 t_j^{\mu-1} + C_2 \sum_{k=1}^{j-1} \int_{t_k}^{t_{k+1}} (t_j-u)^{\gamma-1} z_k \rd u \right) \\
&\leq C_1 t_n^{\mu-1} + C_2 \widetilde{B}(\gamma,\mu) C_1 t_n^{\gamma+\mu-1} +
C_2^2 \widehat{B}(\gamma,\gamma) \sum_{j=1}^{n-2} \int_{t_j}^{t_{j+1}} (t_n-s)^{2\gamma-1} z_j \rd s\\
&\leq \cdots\\
&\leq C_1 t_n^{\mu-1} + \sum_{k = 1}^{m} \left( \prod_{i = 1}^{k-1} \widehat{B}(i\gamma,\gamma) \right) C_2^k \widetilde{B}(k\gamma,\mu) C_1 t_n^{k\gamma+\mu-1} \\
&\quad + C_2^{m+1} \left( \prod_{i = 1}^m \widehat{B}(i\gamma,\gamma) \right) \cdot \sum_{j = 1}^{n-m-1} \int_{t_j}^{t_{j+1}} (t_n-s)^{(m+1)\gamma-1} z_j \rd s.
\end{align*}
For any $\gamma \in (0,1)$, one can fix a finite integer $m_0 \in (\gamma^{-1}-1, \gamma^{-1}]$ such that $(m+1)\gamma-1 > 0$ for all $m \geq m_0$. When $n=1$, Lemma \ref{lem.Gronwall} is trivial. When $2\leq n\leq m_0$,
\begin{align*}
z_n
&\leq C_1 t_n^{\mu-1} + \sum_{k = 1}^{m_0} \left( \prod_{i = 1}^{k-1} \widehat{B}(i\gamma,\gamma) \right) C_2^k \widetilde{B}(k\gamma,\mu) C_1 t_n^{k\gamma+\mu-1} \\
&\leq \left(1+\sum_{k = 1}^{m_0} \left( \prod_{i = 1}^{k-1} \widehat{B}(i\gamma,\gamma) \right) C_2^k \widetilde{B}(k\gamma,\mu) T^{k\gamma}\right) C_1 t_n^{\mu-1},
\end{align*}
which implies that Lemma \ref{lem.Gronwall} holds for the case $2\leq n\leq m_0$. When $n\ge m_0+1$,
\begin{align*}
z_n
&\leq C_1 t_n^{\mu-1} + \sum_{k = 1}^{m_0} \left( \prod_{i = 1}^{k-1} \widehat{B}(i\gamma,\gamma) \right) C_2^k \widetilde{B}(k\gamma,\mu) C_1 t_n^{k\gamma+\mu-1} \\
&\quad + C_2^{{m_0}+1} \left( \prod_{i = 1}^{m_0} \widehat{B}(i\gamma,\gamma) \right) \sum_{j = 1}^{n-{m_0}-1} \int_{t_j}^{t_{j+1}} (t_n-s)^{({m_0}+1)\gamma-1} z_j \rd s\\
&\leq \left(1+\sum_{k = 1}^{m_0} \left( \prod_{i = 1}^{k-1} \widehat{B}(i\gamma,\gamma) \right) C_2^k \widetilde{B}(k\gamma,\mu) T^{k\gamma}\right) C_1 t_n^{\mu-1}\\
&\quad+C_2^{{m_0}+1} T^{({m_0}+1)\gamma-1}\left( \prod_{i = 1}^{m_0} \widehat{B}(i\gamma,\gamma) \right) \sum_{j = 1}^{n-{m_0}-1} \int_{t_j}^{t_{j+1}} z_j \rd s,
\end{align*}
which in combination with Lemma \ref{lem.GW1} completes the proof.
\end{proof}

\section{An auxiliary lemma}

\begin{lemma}
Let $\alpha \in (0,1)$ and $\beta \in (0,\frac{1}{2})$, $t_n \in \cM_r$ with $n \in \{2,3,\ldots,N\}$ and $r \geq 1$. Then there exists $C = C(\alpha,\beta,r,T)$ such that
\begin{gather} 
\int_{t_1}^{t_n} (t_n-s)^{-2\beta} \int_{0}^{t_1} \big| (s-u)^{-\beta} - (\hat{s}-u)^{-\beta} \big|^2 \rd u \rd s
\leq C N^{-2(1-2\beta)}, \label{eq.diffSqureDoub1} \\ 
\int_{t_1}^{t_n} (t_n-s)^{-2\beta} \int_{t_1}^{\hat{s}} \big| (s-u)^{-\beta} - (\hat{s}-u)^{-\beta} \big|^2 \hat{u}^{2\beta - 1} (u-\hat{u})^{1-2\beta} \rd u \rd s
\leq C N^{-2(1-2\beta)}, \label{eq.diffSqureDoub2} \\ 
\int_{0}^{t_{n-1}} \left| \int_{\check{u}}^{t_{n}} (t_n-s)^{-\alpha} \big( (s-u)^{-\beta} - (\hat{s}-u)^{-\beta} \big) \rd s \right|^2 \rd u
\leq C N^{-2\min\{\frac{3}{2}-\alpha-\beta,\, 1-2\beta\}}. \label{lem.inA2}
\end{gather} 
\end{lemma}

\begin{proof}
Following the proof of Lemma 3.2 of \cite{DaiXiaoBu2021}, one can read 
\begin{align}\label{eq:tn1}
\int_{t_1}^{t_2} (t_n-s)^{-2\beta} \int_{0}^{\hat{s}} \big| (s-u)^{-\beta} - (\hat{s}-u)^{-\beta} \big|^2 \rd u \rd s 
\leq C N^{-2(1-2\beta)}.
\end{align}
For any $s\in[t_2,T]$ and $u\in(0,t_1)$, 
\begin{align*}
\big| (s-u)^{-\beta} - (\hat{s}-u)^{-\beta} \big|^2 
&\leq C \Big| \int_{\hat{s}}^s (v-u)^{-\beta-1} \rd v \Big|^2 \\
&\leq C (\hat s-u)^{-1} \big| (s-u)^{\frac{1}{2}-\beta} - (\hat{s}-u)^{\frac{1}{2}-\beta} \big|^2\\
&\leq C(s-t_2)^{2\beta-1} h_2^{-2\beta} (s-\hat{s})^{2(\frac{1}{2}-\beta)} \\
&\leq C N^{-(1-2\beta)}h_2^{-2\beta} (s-t_2)^{2\beta-1}. 
\end{align*}
Then by \eqref{eq.BetaInte}, 
\begin{align*}
&\quad\, \int_{t_2}^{t_n} (t_n-s)^{-2\beta} \int_{0}^{t_1} \big| (s-u)^{-\beta} - (\hat{s}-u)^{-\beta} \big|^2 \rd u \rd s \\
&\leq C N^{-(1-2\beta)} h_2^{-2\beta} h_1\int_{t_2}^{t_n} (t_n-s)^{-2\beta} (s-t_2)^{2\beta-1} \rd s \leq C N^{-2(1-2\beta)}, 
\end{align*}
which along with \eqref{eq:tn1} yields \eqref{eq.diffSqureDoub1}. 

By \eqref{eq.s} and \eqref{eq.BetaInte},
\begin{align*}
&\quad\, \int_{t_1}^{\hat{s}} \big| (s-u)^{-\beta} - (\hat{s}-u)^{-\beta} \big|^2 \hat{u}^{2\beta - 1} (u-\hat{u})^{1-2\beta} \rd u \\
& \leq C N^{-(1-2\beta)} \int_{t_1}^{\hat{s}} \big| (s-u)^{-\beta} - (\hat{s}-u)^{-\beta} \big|^2
u^{2\beta-1} \rd u \\
& \leq C N^{-(1-2\beta)} \int_{0}^{\hat{s}} (\hat{s}-u)^{-2\beta} u^{2\beta-1} - (s-u)^{-2\beta} u^{2\beta-1} \rd u\\
& \leq C N^{-(1-2\beta)} \int_{\hat{s}}^s (s-u)^{-2\beta} u^{2\beta-1} \rd u \leq C N^{-2(1-2\beta)} \hat{s}^{2\beta-1}.
\end{align*}
Then \eqref{eq.diffSqureDoub2} follows from the above inequality and \eqref{eq.BetaC}.

Based on $(\hat{s}-u)^{-\beta}-(s-u)^{-\beta} = \beta \int_{\hat{s}}^s (v-u)^{-\beta-1} \rd v$, by denoting
\begin{align*}
\cI_1 &:= \sum_{i=0}^{n-3} \int_{t_i}^{t_{i+1}} \Big| \int_{t_{i+1}}^{t_{i+2}} (t_n-s)^{-\alpha} \Big( \int_{t_{i+1}}^s (v-u)^{-\beta-1} \rd v \Big) \rd s \Big|^2 \rd u, \\
\cI_2 &:= \int_{t_{n-2}}^{t_{n-1}} \Big| \int_{t_{n-1}}^{t_{n}} (t_n-s)^{-\alpha} \Big( \int_{t_{n-1}}^s (v-u)^{-\beta-1} \rd v \Big) \rd s \Big|^2 \rd u, \\
\cI_3 &:= \sum_{i=0}^{n-2} \int_{t_i}^{t_{i+1}} \Big| \int_{t_{i+2}}^{t_{n}} (t_n-s)^{-\alpha} \Big( \int_{\hat{s}}^s (v-u)^{-\beta-1} \rd v \Big) \rd s \Big|^2 \rd u,
\end{align*}
the proof of \eqref{lem.inA2} suffices to show that $\cI_{l}\le C N^{-2\min\{\frac{3}{2}-\alpha-\beta,\, 1-2\beta\}}$ for $l=1,2,3$.

When $n = 2$, $\cI_{1} = 0$. While $n \geq 3$,
\begin{align*}
\cI_{1}
&\leq \sum_{i=0}^{n-3} \int_{t_{i}}^{t_{i+1}} \Big| \int_{t_{i+1}}^{t_{i+2}} (t_n-s)^{-\alpha} \Big( \int_{t_{i+1}}^s (t_{i+1}-u)^{-\frac{\beta}{2}-\frac{1}{4}} (v-u)^{-\frac{\beta}{2}-\frac{3}{4}} \rd v \Big) \rd s \Big|^2 \rd u \\
&\leq C N^{-(\frac{1}{2}-\beta)} \sum_{i=0}^{n-3} \int_{t_{i}}^{t_{i+1}} \Big| \int_{t_{i+1}}^{t_{i+2}} (t_n-s)^{-\alpha} \rd s \Big|^2 (t_{i+1}-u)^{-\frac{1}{2}-\beta} \rd u \\
&\leq C N^{-(1-2\beta)} \sum_{i=0}^{n-3} \Big| \int_{t_{i+1}}^{t_{i+2}} (t_n-s)^{-\alpha} \rd s \Big|^2\\
&\le CN^{-(1-2\beta)}\sum_{i=0}^{n-3} h_{i+2} \int_{t_{i+1}}^{t_{i+2}} (t_n-s)^{-2\alpha} \rd s \le C N^{-2\min\{\frac{3}{2}-\alpha-\beta,\, 1-2\beta\}}.
\end{align*}
For $\cI_2$, it follows from \eqref{eq.BetaInte} that
\begin{align*}
\cI_{2}
&= \int_{t_{n-2}}^{t_{n-1}} \Big| \int_{t_{n-1}}^{t_{n}} \int_{v}^{t_n} (t_n-s)^{-\alpha} (v-u)^{-\beta-1} \rd s \rd v \Big|^2 \rd u \\
&\leq C \int_{t_{n-2}}^{t_{n-1}} \Big| \int_{t_{n-1}}^{t_{n}} (t_n-v)^{1-\alpha} (v-u)^{-\beta-1} \rd v \Big|^2 \rd u \\
&\leq C \int_{t_{n-2}}^{t_{n-1}} \Big| \int_{t_{n-1}}^{t_{n}} (t_n-v)^{1-\alpha} (v-t_{n-1})^{-\frac{\beta}{2}-\frac{3}{4}} (t_{n-1}-u)^{-\frac{\beta}{2}-\frac{1}{4}} \rd v \Big|^2 \rd u \\
&\leq C N^{-(\frac{5}{2}-2\alpha-\beta)} \int_{t_{n-2}}^{t_{n-1}} (t_{n-1}-u)^{-\beta-\frac{1}{2}} \rd u \leq C N^{-2(\frac{3}{2}-\alpha-\beta)}.
\end{align*}
Finally, it remains to bound $\cI_3$. When $\alpha \in (0, \frac{1}{2}+\beta)$, by \eqref{eq.singleSize},
\begin{align*}
\cI_3
&\leq C \sum_{i=0}^{n-2} \int_{t_i}^{t_{i+1}} \Big| \int_{t_{i+2}}^{t_{n}} (t_n-s)^{-\alpha} \Big( \int_{\hat{s}}^s (s-u)^{\beta-1} (v-u)^{-2\beta} \rd v \Big) \rd s \Big|^2 \rd u \\
&\leq C N^{-2(1-2\beta)} \sum_{i=0}^{n-2} \int_{t_i}^{t_{i+1}} \Big| \int_{u}^{t_{n}} (t_n-s)^{-\alpha} (s-u)^{\beta-1} \rd s \Big|^2 \rd u \\
&\leq C N^{-2(1-2\beta)} \int_{0}^{t_{n-1}} (t_n-u)^{-(2\alpha-2\beta)} \rd u \leq C N^{-2(1-2\beta)}.
\end{align*}
When $\alpha \in [\frac{1}{2}+\beta,1)$, by \eqref{eq.singleSize},
\begin{align*}
\cI_3
&\leq C \sum_{i=0}^{n-2} \int_{t_i}^{t_{i+1}} \Big| \int_{t_{i+2}}^{t_{n}} (t_n-s)^{-\alpha} \Big( \int_{\hat{s}}^s (s-u)^{\frac{\beta}{2}-1} (v-u)^{-\frac{3\beta}{2}} \rd v \Big) \rd s \Big|^2 \rd u \\
&\leq C h_{n}^{2-3\beta} \int_{0}^{t_{n-1}} \Big| \int_{u}^{t_{n}} (t_n-s)^{-\alpha} (s-u)^{\frac{\beta}{2}-1} \rd s \Big|^2 \rd u \\
&\leq C h_{n}^{2-3\beta} \int_{0}^{t_{n-1}} (t_n-u)^{-(2\alpha-\beta)} \rd u\leq C N^{-2(\frac{3}{2}-\alpha-\beta)}.
\end{align*}
The proof is completed. 
\end{proof}

\providecommand{\bysame}{\leavevmode\hbox to3em{\hrulefill}\thinspace}
\providecommand{\MR}{\relax\ifhmode\unskip\space\fi MR }
\providecommand{\MRhref}[2]{%
  \href{http://www.ams.org/mathscinet-getitem?mr=#1}{#2}
}
\providecommand{\href}[2]{#2}


\end{document}